\documentclass[smallextended]{svjour3}       
\usepackage{graphicx}
\usepackage{amsmath,amsfonts,amssymb,amscd,amsopn}
\usepackage{mathrsfs}  
\usepackage{color}
\usepackage{xspace}
\usepackage{url}

\newcommand{\BM}[1]{\color{black}{#1}\color{black}\xspace}
\newenvironment{BMe}{\color{red}}{} 

\def\BMe{\color{black}}
\def\Ascr{\mathscr{A}}

\begin{document}

 \newtheorem{thm}{Theorem}
 \numberwithin{thm}{section}
 \newtheorem{cor}[thm]{Corollary}
 \newtheorem{lem}[thm]{Lemma}{\rm}
 \newtheorem{prop}[thm]{Proposition}

 \newtheorem{defn}[thm]{Definition}{\rm}
 \newtheorem{assumption}[thm]{Assumption}
 \newtheorem{rem}[thm]{Remark}
 \newtheorem{ex}[thm]{Example}
\numberwithin{equation}{section}

 \def\im{\mathbf{i}}
\def\la{\langle}
\def\ra{\rangle}
\def\glexe{\leq_{gl}\,}
\def\glex{<_{gl}\,}
\def\e{{\rm e}}

\def\x{\mathbf{x}}
\def\P{\mathbf{P}}
\def\h{\mathbf{h}}
\def\by{\mathbf{y}}
\def\bz{\mathbf{z}}
\def\F{\mathcal{F}}
\def\R{\mathbb{R}}
\def\T{\mathbf{T}}
\def\N{\mathbb{N}}
\def\D{\mathbf{D}}
\def\V{\mathbf{V}}
\def\U{\mathbf{U}}
\def\K{\mathbf{K}}
\def\Q{\mathcbf{Q}}
\def\H{\mathbf{H}}
\def\M{\mathscr{M}}
\def\oM{\overline{\mathscr{M}}}
\def\O{\mathbf{O}}
\def\C{\mathbb{C}}
\def\P{\mathbf{P}}
\def\Z{\mathbb{Z}}
\def\H{\mathcal{H}}
\def\A{\mathbf{A}}
\def\V{\mathbf{V}}
\def\AA{\overline{\mathbf{A}}}
\def\B{\mathbf{B}}
\def\c{\mathbf{C}}
\def\L{\mathcal{L}}
\def\Ec{\mathcal{E}}
\def\bS{\mathbf{S}}
\def\H{\mathcal{H}}
\def\I{\mathbf{I}}
\def\Y{\mathbf{Y}}
\def\X{\mathbf{X}}
\def\G{\mathbf{G}}
\def\f{\mathbf{f}}
\def\z{\mathbf{z}}
\def\v{\mathbf{v}}
\def\y{\mathbf{y}}
\def\d{\hat{d}}
\def\bx{\mathbf{x}}
\def\bI{\mathbf{I}}
\def\y{\mathbf{y}}
\def\g{\mathbf{g}}
\def\w{\mathbf{w}}
\def\b{\mathbf{b}}
\def\a{\mathbf{a}}
\def\u{\mathbf{u}}
\def\q{\mathbf{q}}
\def\e{\mathbf{e}}
\def\s{\mathcal{S}}
\def\cc{\mathcal{C}}
\def\co{{\rm co}\,}
\def\tg{\tilde{g}}
\def\tx{\tilde{\x}}
\def\tg{\tilde{g}}
\def\tA{\tilde{\A}}

\def\fc{\mathfrak{d}}
\def\supmu{{\rm supp}\,\mu}
\def\supp{{\rm supp}\,}
\def\cd{\mathcal{C}_d}
\def\cok{\mathcal{C}_{\K}}
\def\cop{COP}
\def\vol{{\rm vol}\,}
\def\blue{\color{blue}}
\def\red{\color{red}}
\def\bo{\boldsymbol{\lambda}}
\def\T{\mathbb{T}}
\def\AA{\mathbb{A}}
\def\CC{\mathbb{C}}
\def\NN{\mathbb{N}}
\def\RR{\mathbb{R}}
\def\zb{\mathbf{z}}
\def\im{\mathbf{i}}

\title{Sparse polynomial interpolation: sparse recovery, super resolution, or Prony?
}


\author{C\'edric Josz         \and
        Jean Bernard Lasserre \and 
Bernard Mourrain
}




\institute{C\'edric Josz \at
              LAAS-CNRS, 7 avenue du Colonel Roche, BP 54200 \\
	     31031 Toulouse C\'edex 4, France\\
              \email{cedric.josz@gmail.com}           
           \and
           Jean Bernard Lasserre \at
           LAAS-CNRS and Institute of Mathematics\\
           7 avenue du Colonel Roche, BP 54200\\
           31031 Toulouse C\'edex 4, France
           \email{lasserre@laas.fr}
	\and
	Bernard Mourrain\\
INRIA, 2004 route des Lucioles\\
 06902 Sophia Antipolis, France \\
\email{Bernard.mourrain@inria.fr}
}

\date{}
\maketitle

\begin{abstract}
We show that the sparse polynomial interpolation problem reduces to
a discrete super-resolution problem on the $n$-dimensional torus. Therefore the semidefinite programming 
approach initiated by Cand\`es \& Fernandez-Granda \cite{candes_towards_2014} 
in the univariate case can be applied. We extend their result to the multivariate case, i.e., we show that exact recovery is guaranteed provided that a geometric spacing condition on the €œsupports€ holds and the number of evaluations are sufficiently many (but not many). It also turns out that
the sparse recovery LP-formulation  of $\ell_1$-norm minimization 
is also guaranteed to provide exact recovery {\it provided that} the
evaluations are made in a certain manner and even though the
Restricted Isometry Property for exact recovery is not satisfied. (A
naive sparse recovery LP-approach does not offer such a guarantee.)
Finally we also describe the algebraic Prony method for sparse interpolation, which also recovers the exact decomposition but from  
less point evaluations and with no geometric spacing condition. We provide two sets of numerical experiments, one in which the super-resolution technique and Prony's method seem to cope equally well with noise, and another in which the super-resolution technique seems to cope with noise better than Prony's method, at the cost of an extra computational burden (i.e. a semidefinite optimization).
\keywords{Linear programming \and Prony's method \and Semidefinite programming \and super-resolution}
\end{abstract}

\section{Introduction}
\label{intro}
In many domains, functions can be described in a way which is easy to evaluate, but not
necessarily easy to identify. This can be the case when the function
comes from the analysis of the input-ouput response of a complex
system or from an algorithmic construction.
Interpolation strategies have shown to be very effective in the
reconstruction of such {\it black-box} functions, in particular in
computer algebra, for sparse multivariate polynomials. 
Such black-box polynomials may be built with ``approximate''
coefficients, so that the evaluation at a point may be an approximate
value with error or noise. While efficient exact methods exist for
the interpolation of sparse polynomials, interpolation of approximate
sparse multivariate polynomials remains a challenging problem.

A motivation of this work is to show that the sparse
interpolation problem can be solved exactly under some conditions of
separability of the support, by three different methods following
different perspectives. One of them is a direct algebraic method and
the two others are based on convex optimization tools (LP, SDP). 
We analyze them in detail and investigate their numerical robustness and their efficiency to address the
interpolation problem of black-box sparse multivariate polynomials.

Suppose that we are given a {\it black-box} polynomial $g\in\R[\z]$,
that is, $g$ is unknown but given any ``input'' point $\z\in\C^n$, the black-box outputs the complex number $g(\z)$.
{We assume that the polynomial $\z\mapsto
g(\z)=\sum_{\alpha} g_{\alpha} \z^{\alpha}$ is {\it sparse}, that is,
it has only a few number $r$ of non-zero terms, compared to the number of monomials of degree less or equal to the degree of $g$}.
{\it Sparse interpolation} is concerned with recovering the unknown monomials
$(\z^{\alpha})$ and coefficients $(g_\alpha)$ of a sparse polynomial
{in a way, which depends on the number $r$ of non-zero terms of $g$}. In the sole knowledge of a few (and as few as possible) values of $g$ at some points
$(\z_k)\subset\C^n$ that one may choose at our convenience, {one want
to recover the $r$ non-zero terms of the polynomial $g$}.

\BM{Hereafter, we present three families of methods
for robust sparse interpolation, using either direct algebraic computation or
convex optimization. Direct algebraic methods shall compute sparse representation using a
minimal number of values. Convex optimization techniques shall help improving
robustness, in the presence of numerical errors. All of these methods allow to choose the points of evaluation.
}

\subsection*{Prony}
The method goes back to the pioneer  work of G. R de Prony \cite{baron_de_prony_essai_1795} who was interested in recovering  a sum of few exponential terms from sampled values of the function. Thus Prony's method
is also a standard tool to recover a complex atomic measure from
knowledge of some of its moments \cite{kunis_multivariate_2016}. Briefly, in the univariate setting
this purely algebraic method consists of two steps: (i) Computing 
the coefficients of a polynomial $p$ whose roots form the finite
support of the unknown measure. As $p$ satisfies a recurrence relation 
it is the unique element (up to scaling) in the kernel of a (Hankel) matrix. (ii) The weights associated to the atoms of the support
solve a Vandermonde system.

This algebraic method has then been used in the context of sparse polynomial interpolation.
In the univariate case it consists
in evaluating the black-box polynomial at values of the form $\varphi^{k}$ for a finite number of  
pairs $(k,\varphi)\in \N\times\C$, fixed.
A sequence of $2\,r$ evaluations allows to recover the decomposition exactly, where $r$ is the number of terms of the sparse polynomial.
The decomposition is obtained by computing a minimal recurrence relation between these evaluations, by finding the roots of the associate polynomial, which yields the exponents of the monomials and by solving a Vandermonde system which yields the coefficients of the terms in the sparse polynomial.

Since then, it has been extended to address numerical issues and to
treat applications in various domains, particularly in signal processing.
\BM{Methods such as MUSIC, ESPRIT extend the initial method of Prony,
by adding robust numerical linear algebra ingredients}.
See e.g. {\cite{roy_esprit-estimation_1989}}, {\cite{swindlehurst_performance_1992}},
{\cite{golub_separable_2003}}, {\cite{beylkin_approximation_2005}}, \cite{potts_nonlinear_2011}, \cite{pereyra_exponential_2012} and the many references therein. 

The approach is closely related to sparse Fast Fourier Transform
techniques, where evaluations at powers of the $N$-th root of unity
are used to recover a $r$-sparse signal. The bounds, in the univariate
case, on the number of samples and runtime complexity are linear in $r$
up to polylog factors in $N$ or $r$. See e.g. \cite{CandesNearOptimalSignalRecovery2006}, \cite{HassaniehNearlyoptimalsparse2012}.

From an algorithmic point of view, the approach has been improved
by exploiting the Berlekamp-Massey algorithm \cite{berlekamp_nonbinary_1968},  \cite{massey_shift-register_1969} and the structure of the involved matrices; see e.g. \cite{kaltofen_improved_1989}, \cite{zippel_interpolating_1990}, \cite{grigoriev_fast_1990}.

Prony's approach has also been applied to treat multivariate sparse
interpolation problems \cite{zippel_probabilistic_1979} and
\cite{ben-or_deterministic_1988}\BM{, by evaluation at points with
coordinates in geometric progressions}.
It has also been extended  to approximate data
\cite{giesbrecht_symbolicnumeric_2009}\BM{, using the same type of
  point evaluation sequences}. It has been 
applied to sparse polynomial interpolation 
with noisy data in \cite{comer2012} 
to provide a way to recover a 
blackbox univariate polynomial exactly when some (but not all) of its evaluations are corrupted with noise (in the sipirit of error-decoding). 
 
Generalizations of Prony's method to multivariate reconstruction
problems have been developed more recently. 
\BM{In \cite{kunis_multivariate_2016}, a projection based method is used to
compute univariate polynomials which roots determine the coordinates
of the terms in the sparse representation.
In \cite{sauer_pronys_2016-1}, an $H$-basis of the
kernel ideal of a moment matrix is computed and used to find the roots
which determine the sparse decomposition.
Direct decomposition methods which compute
the algebraic structure of the  Artinian Gorenstein
algebra associated to the moment matrix and deduce the sparse
representation from eigenvectors of multiplication operators
are developed in \cite{mourrain_polynomial-exponential_2016} and
\cite{harmouch_structured_2017}.
}

\subsection*{Sparse recovery}
``Sparse recovery'' refers to methods for estimating a sparse
representation from solutions of underdetermined linear systems.
It corresponds to the mathematical aspects of what is know as ``Compressed sensing'' in Signal Processing.
Here we consider a naive ``sparse recovery" LP-approach, which consists of solving
$\min\{\Vert\x\Vert_1:\A\x=b\}$ where $\x$ is the vector of
coefficients of the unknown polynomial and $\A\x=b$ are linear constraints obtained from 
evaluations at some given points. By minimizing the $\ell_1$ norm one
expects to obtain a 
``sparse" solution to the undetermined system $\A\x=b$. However 
since the matrix $\A$ does not satisfy the sufficient {\it Restricted Isometry Property} (RIP), exact  recovery 
is not guaranteed (at least by invoking results from compressed sensing).
\BM{Only probabilistic results may be expected in the univariate case
  if enough sampling points (of the order $\mu\,r\,log(d)$ where $r$
  is number of non-zero terms, $d$ is the maximal degree of the terms and $\mu>1$ is a
constant measuring the coherence of $\A$) are
  chosen at random, as in \cite{CandesProbabilisticRIPlessTheory2011}.}

\subsection*{Super-resolution}

\BM{{\it Super-resolution} refers to techniques to enhancing the
  resolution of sensing systems. In \cite{candes_towards_2014}, it
  refers to the process or retrieving fine scale
structures from coarse scale information, such as Fourier coefficients.
In more mathematical terms, it} consists in recovering the
  support of a sparse atomic (signed) measure on a compact set $K$,
  from known moments. Hereafter we will consider the particular case 
  where $K$ is the multi-dimensional torus $\T^n\subset\C^n$.  In the work of
  Cand\`es and Fernandez-Granda \cite{candes_towards_2014} it is
  shown that if the support atoms of the measure are well-separated then the
  measure is the unique solution of an infinite-dimensional convex optimization problem on a space of measures
  with the  total variation as minimization criterion. In the univariate case its (truncated) moment matrix can be recovered
  by solving a single {\it Semidefinite Program}   (SDP).
  The number of evaluations needed for exact recovery is then at most $4s$ if $s$ is the number of atoms\footnote{As noted in Cand\`es and Fernandez-Granda \cite{candes_towards_2014}, with the proviso that 
      the number of evaluations is larger than 128} (and in fact significantly less in all numerical examples provided).
  Interestingly, the total-variation minimization technique adapts nicely to noisy model and yields stable approximations of the weighted sum of Dirac measures, provided that the support atoms are well separated, see e.g. \cite{candes_super-resolution_2013}, \cite{azais_spike_2015}, \cite{duval_exact_2015}.

An extension to the multivariate case has been proposed in \cite{de_castro_exact_2017} to recover weight sums of Dirac measures in $\R^{n}$, where now one needs to solve a {\it hierarchy} of semidefinite programs.

The existence and unicity of the solution of the total variation minimization problem relies on the
  existence of a dual certificate, that is, a polynomial with $\sup$-norm reached at the points of the support of the measure. 
  The relaxation into a hierarchy of semidefinite programs \cite{de_castro_exact_2017} yields a decomposition into a finite weighted  sum of Dirac measures, provided that at some order of the hierarchy, a flat extension condition is satisfied at an optimal solution. Then the decomposition can  be recovered from the moment matrix at this optimal solution by applying a Prony's like technique.

\subsection*{Contribution}\ \vspace{-0.5cm}

$\bullet$ We propose a new multivariate variant of Prony's method for
sparse polynomial interpolation\BM{, which avoids projections in one
 variables and requires a small number of evaluations. In particular,
it differs from approaches such as
\cite{giesbrecht_symbolicnumeric_2009}, which uses special ``aligned'' moments to
apply Prony univariate method.}
In the univariate case, the new method only requires $r+1$
evaluations (instead of $2r$) where $r$ is the number of monomials of 
the blackbox polynomial. \BM{Similarly in the multivariate case, we
  show that the number of needed evaluations is significantly reduced.}
It involves a Toeplitz matrix rather than a 
Hankel matrix. Numerical experiments confirm the theoretical result 
regarding the number of evaluations \BM{and robustness against
  perturbations. This new multivariate Toeplitz-Prony can be seen as
  an extension of ESPRIT methods in several variables. As stated in
\cite{StoicaSpectralanalysissignals2005}[p.167], ESPRIT should be
preferred to MUSIC for frequency estimation in signal processing.
The numerical experiments corroborate this claim by showing the
good numerical behavior of the new multivariate Toeplitz-Prony method.  
}

$\bullet$ We consider the naive  sparse recovery approach 
to sparse interpolation via $\ell_1$-norm minimization $\min\{\Vert\x\Vert_1:\A\x=b\}$ and we characterize optimal solutions
via standard arguments of linear programming (LP). Interestingly, 
this characterization is a ``formal analogue" in appropriate spaces of  that in super-resolution (\ref{super-resolution}) 
in some measure spaces.
However as the matrix $\A$ does not satisfy the RIP
there is no guarantee (at least by invoking results from compressed sensing) that an optimal solution is
unique and corresponds to the unique sparse black-box polynomial $g$.

$\bullet$ 
We then propose another approach which uses the fact that one has the choice of points $(\zeta_k)\subset\C^n$ at which evaluations
of $g$ can be done through the black-box and yields the following simple but crucial observation:
By choosing $\zeta_k$ as some power $\varphi^\beta$ with $\beta\in\N^n$
(and where  $\varphi\in\T^n$ is fixed, arbitrary) the sparse polynomial $p$ can be viewed as a signed atomic measure $\mu$ on $\T^n$
with finitely many atoms $(\varphi^\alpha)\subset\T^n$ associated with the nonzero coefficients of $p$ (the signed weights of $\mu$
associated with each atom).
In doing so the sparse interpolation problem is completely equivalent
to a super-resolution problem on the multi-dimensional torus $\T^n$.
We prove a new unicity theorem (Theorem \ref{thm:unicity}) for the optimal
solution of the super-resolution problem on $\T^{n}$, provided enough moments
of the measure are known. Namely, if $\fc\ge 4 \pi
r (r-1) \Ec(\Xi)$ where  $\Ec(\Xi)$ is the interpolation entropy of the
support $\Xi$ of a measure $\mu$ on $\T^{n}$ and $r= \vert \Xi \vert$, we show that the
super-resolution optimisation problem with all the moments of $L_{1}$-degree
$\ge \fc$ of the measure $\mu$ has a unique solution (i.e. $\mu$). 
This result is  non trivial extension to any dimension of Theorems 1.2,
1.3 in Candès and Fernandez-Granda  \cite{candes_towards_2014} proved for one and two variables.
\BMe
Consequently, the sparse polynomial is the unique optimal solution of a certain infinite-dimensional linear program on a space of measures, provided that a geometric condition of minimum spacing (between atoms of the support) is satisfied
and sufficiently many evaluations are available. Notice that previous works on Prony's method (e.g. \cite{comer2012}) have also exploited 
(but in a different manner) evaluations at consecutive powers of a fixed element. In fact our view of a polynomial as a signed atomic measure 
on the torus is probably the shortest way to explain why Prony's method can be used for polynomial interpolation 
(as the original Prony's method can be interpreted directly as reconstructing an atomic measure on the complex plane from 
some of its moments \cite{kunis_multivariate_2016}).

We then relax this problem to a \BM{new} hierarchy of semidefinite
programs. \BM{This hierarchy requires less moments or evaluations than
  a degree-based moment relaxation.}
In principle, the convergence is only asymptotic (and guaranteed to
be finite only in the cases $n=1, 2$). However generic finite
convergence results of polynomial optimization of 
  \cite{Nie} seem to be also valid in our context as evidenced from our numerical
  experiments (and in those in De Castro et al. \cite{de_castro_exact_2017} in signal processing).
The flat extension rank condition on moment matrices of Curto and Fialkow 
\cite[Theorem 1.1]{curto}
or its generalization in \cite{laurent_generalized_2009}
can be extended to Toeplitz-like moment matrices \cite{cedric},
to test whether finite convergence takes place.
  In all our numerical experiments, finite
  convergence takes place and the coefficients and exponents of the
  unknown polynomial could be extracted. To give an idea, a univariate
  polynomial of degree $100$ with 3 atoms can be recovered by solving
  a single SDP with associated $4\times 4$ Toeplitz matrices and
  which only involves $4$ evaluations. On the other hand if some atoms are close
  to each other then more information (i.e. evaluations) is needed as
  predicted by the spacing condition (and confirmed in some numerical
  experiments).

$\bullet$ In practice we reduce the number of measurements (i.e., evaluations) needed to retrieve a sparse polynomial when using super resolution. To do this we invoke a result (Lemma \ref{lemma:atom} of this paper)
related to the full complex moment problem. It states that atomic measures on $\C^n$ with finitely many atoms are completely characterized by their moments $(\int\z^\alpha\,d\mu)_{\alpha\in\N^n}$,
with no need of all moments $(\int \bar{\z}^\beta\z^\alpha\,d\mu)_{\alpha,\beta}$ involving conjugates.
This result, which holds true in full generality, yields a simplified hierarchy with significant computational savings. It is the subject of future work to determine whether this preserves the guarantee of ``asymptotic" recovery of the original complete hierarchy; in our numerical experiments, finite convergence is always observed, and with fewer measurements than in the original method.

$\bullet$ A rigorous LP-approach. In fact the interpolation problem is even a {\it discrete} super-resolution problem
(i.e. recovery of a discrete signal) where the atomic measure consists of finitely many atoms on a fixed grid $\{t/N\}_{t=0,\ldots,N-1}$ as described in Cand\`es and Fernandez-Granda 
\cite[\S 1.4]{candes_towards_2014}. Therefore, in view of our new
uniqueness result for $n>2$, this fact also validates exact recovery 
via  a (sparse recovery) LP-formulation of $\ell_1$-minimization 
$\min_\x\{\Vert\x\Vert_1: \A\x=b\}$
{\it provided} that the spacing condition is satisfied and evaluations  (modeled by the constraints $\A\x=b$ are made in a certain manner on the torus $\T^n$, and not
on a random sample of points in $\R^n$). Interestingly, this provides us with an important  case of sparse recovery
where exact recovery is guaranteed even though the RIP property is not satisfied. However from a practical side the SDP formulation is
more efficient and elegant. Indeed for instance in the univariate case 
the size of the Toeplitz matrix involved is directly related to the number of 
atoms to recover whereas in the sparse recovery LP-approach, one has to
fix \textit{a priori} the length $N$ of the vector $\x$ (which depends on the degree of the unknown polynomial, possibly very large) even if 
ultimately one is interested only in its few non zero entries (usually a very small number).

$\bullet$ Finally we provide a numerical comparison of the three approaches (LP, SDP and Prony) 
on a  sample of problems and comment on their respective advantages and drawbacks. We clarify the relationship between Prony's method and super-resolution. In \cite{candes_towards_2014} Prony's method was briefly mentioned and neglected as sensitive to noise in the data (in contrast to super-resolution). We try to clarify this statement: actually, super-resolution \textit{requires} Prony's method (or some variant of it) to extract relevant information from the output (the optimal solution) of the semidefinite program. In other words, super-resolution preprocesses the input data to Prony's method via a convex optimization procedure. We find that this sometimes helps to deal with noise in the context of polynomial interpolation, confirming the elegant theory of \cite{candes_towards_2014}. In some instances, super resolution does not perform well because of numerical issues present in current semidefinite programming solvers. To the best of our knowledge this drawback has not been discussed in 
the literature.

\section{Notation, definitions and Preliminary results}
\subsection{Notation and definitions}
Let $\R[\x]$ (resp. $\R[\x]_d$) denote the ring of real polynomials in the variables
$\x=(x_1,\ldots,x_n)$ (resp. polynomials of degree at most $d$), whereas $\Sigma[\x]$ (resp. $\Sigma[\x]_d$) denotes 
its subset of sums of squares (SOS) polynomials (resp. of SOS of degree at most $2d$).
For every
$\alpha=(\alpha_{1},\ldots, \alpha_{n})\in\N^n$ the notation $\x^\alpha$ stands for the monomial $x_1^{\alpha_1}\cdots x_n^{\alpha_n}$ and for every $i\in\N$, let $\N^{p}_d:=\{\beta\in\N^n:\sum_j\beta_j\leq d\}$ whose cardinal is $s(d)={n+d\choose n}$.
A polynomial $f\in\R[\x]$ is written $f=\,\sum_{\alpha\in\N^n}\,f_\alpha\,\x^\alpha$ with $f_{\alpha}$ almost all equal to zero,
and $f$ can be identified with its vector of coefficients $\f=(f_\alpha)$ in the canonical basis $(\x^\alpha)$, $\alpha\in\N^n$.

Denote by $\R[\x]_d^*$ the space of linear functionals on $\R[\x]_d$, identified with $\R^{s(d)}$.
For a closed set $\K\subset\R^n$ denote by $C_d(\K)\subset\R[\x]_d$ the convex cone of polynomials of degree at most $d$ that are nonnegative on $\K$, and for $f\in\R[\x]_d$, let
\[\Vert f\Vert_1\,:=\,\Vert \f\Vert_1\,=\,\displaystyle\sum_{\alpha\in\N^n_{d}}\vert f_\alpha\vert.\]

Denote by $\mathcal{S}^t\subset\R^{t\times t}$ the space of real symmetric matrices,
and for any $\A\in\mathcal{S}^t$ the notation $\A\succeq0$ stands for $\A$ is positive semidefinite.

A real sequence $\mathbf{\sigma}=(\sigma_\alpha)$, $\alpha\in\N^n$, has a {\it representing measure} supported on a set $S\subset\R^n$ if
there exists some finite Borel measure $\mu$ on $S$ such that 
\[\sigma_\alpha\,=\,\int_{S}\x^\alpha\,d\mu(\x),\qquad\forall\,\alpha\in\N^n.\]

The space of finite Borel (signed) measures (resp. continuous functions) on
 $S\subset\R^n$ is denoted by $\M(S)$ (resp. $\mathscr{C}(S)$).

\subsection{Super-resolution}

Let $S\subset\R^n$ and suppose that $\mu$ is a signed atomic measure
supported on a  few atoms $(\z_i)\subset S$, $i=1,\ldots,s$,
i.e., $\mu=\sum_{k=1}^s w_i\,\delta_{\xi_i}$. Super-resolution is concerned with retrieving the 
supports $(\xi_i)\subset S$ as well as the weights $(w_i)\subset\R$, from the sole knowledge of a few (and as few as possible) 
``moments" $(\sigma_k=\int_S g_k\,d\mu)$, $k=1,\ldots,m$, for some functions $(g_k)$. One possible approach is to solve the convex optimization
problem:
\begin{equation}
\label{super-resolution}
\rho\,=\,\displaystyle\inf_{\mu\in \M(S)}\,\{\,\Vert \mu\Vert_{TV}:\quad  \int_S g_k\,d\mu \,=\,\sigma_k,\quad k=1,\ldots,m\,\}\end{equation}
where $\M(S)$ is the space of finite signed Borel measures on $S$ equipped with the total-variation nom 
$\Vert\cdot\Vert_{TV}$. The dual of (\ref{super-resolution}) reads:
\begin{equation}
\label{super-resolution-dual}
\rho^*\,=\,\displaystyle\sup_{\lambda\in \R^m}\,\{\,\sigma^T\lambda: \Vert \sum_{k=1}^m \lambda_k\,g_k\Vert_\infty\,\leq\,1\,\},\end{equation}
where $\Vert f\Vert_\infty=\sup_{\x\in S}\vert f(\x)\vert$. 
(In fact and interestingly, both programs (\ref{super-resolution}) and its dual (\ref{super-resolution-dual}) have already appeared in the sixties
in a convex and elegant formulation of some bang-bang type optimal control problems; see Neustadt \cite{neustadt_optimization_1964} and Krasovskii \cite{krasovskii_theory_1968}.)
The rationale behind this approach is the analogy with {\it sparse recovery}. 
Indeed, the total variation norm $\Vert\mu\Vert_{TV}$ is the analogue for measures
of the $\ell_1$-norm for vectors\footnote{To see this suppose that $\mu$ is the signed atomic measure
$\sum_{i=1}^s \omega_i\delta_{\xi_i}$. Then $\Vert\mu\Vert_{TV}=\Vert \omega\Vert_1$.}. 

In the univariate case when $S$ is an interval (one may also consider the torus $\T\subset\C$)
and the $g_k$'s are the usual algebraic monomials $(\x^k)$, solving (\ref{super-resolution}) then reduces to solving a single semidefinite program (SDP)
and Cand\`es and Fernandez-Granda \cite{candes_towards_2014} have shown that
exact reconstruction is guaranteed provided that
the (unknown) $s$ supports $(z_i)$ are sufficiently spaced and $m\geq \max [4s , 128]$. 

This approach was later generalized to
arbitrary dimension and semi-algebraic sets in De Castro et al \cite{de_castro_exact_2017};  in contrast to
the univariate case, one has to solve a hierarchy of semidefinite programs (instead of a single one).
In the 2-dimensional and 3-dimensional examples treated in \cite{de_castro_exact_2017}, exact recovery is obtained rapidly.

Alternatively one may also recover $\mu$ via the algebraic multivariate Prony method described in \cite{mourrain_polynomial-exponential_2016} and the references therein,
and for which {\it no} minimum geometric separation of the supports is required. In addition, in the univariate case only $m=2 r$ moments are needed for exact recovery.

\subsection{The multivariate Prony method} \label{Prony}
\subsubsection{Hankel Prony}
\label{subsub:Hankel Prony}
A multivariate Prony method has been proposed in
\cite{mourrain_polynomial-exponential_2016,harmouch_structured_2017}\footnote{An
implementation is available at \url{https://gitlab.inria.fr/AlgebraicGeometricModeling/TensorDec.jl}}. We refer to it in this paper as ``Hankel Prony''. It consists in two successive linear algebra operations.\\\\
\textbf{Input} 
\begin{itemize}
\item Measurements $y_{\alpha}\in \C$ for $\alpha=(\alpha_1,\hdots,\alpha_n)\in \mathbb{N}^n$ up to a degree $\|\alpha \|_1 = \sum\limits_{i=1}^n \alpha_k \leqslant d$
\item A threshold $\epsilon > 0$ to determine the numerical rank
\end{itemize}
\textbf{Output} Atomic measure $\mu$
\begin{enumerate}
\item For $d_1 := \lfloor \frac{d}{2} \rfloor $ and $d_2 := \lceil \frac{d}{2} \rceil$ (where $\lfloor \cdot \rfloor$ and $\lceil \cdot \rceil$ denote the ceiling and floor of an integer), a singular value decomposition of a submatrix containing the measurements (i.e. $H_{0}=(y_{\alpha+\beta})_{|\alpha| \leqslant d_1,|\beta|\leqslant d_2-1} = U\Sigma V^*$ where $(\cdot)^*$ stands for adjoint); the threshold $\epsilon > 0$ is used to determine the numerical rank $r$ using the ratio of successive singular values. Precisely, the singular values in the diagonal matrix $\Sigma$ are sorted in decreasing order and the rank is taken to be equal to the first instance when the ratio drops below the threshold. Multiplication matrices of size $r \times r$ can then be formed for each variable, i.e. $M_{k} = \Sigma_r^{-1} U_r H_{k} V_r$ where $H_{k}= (y_{\alpha+\beta+e_k})_{|\alpha| \leqslant d_1,|\beta|\leqslant d_2-1}$, $\Sigma_r$ contains the $r$ greatest singular values in its diagonal, $U_r$ is composed of the first $r$ rows of the conjugate transpose of $U$, $V_r$ is composed of the first $r$ columns of the conjugate of $V$, and $e_k$ denotes the row vector of size $n$ of all zeros apart from $1$ in position $k$.\\
\item An eigen-decompositon of a random linear combination of the multiplication matrices $\sum\limits_{k=1}^n \lambda_k M_{k} = PDP^{-1}$ (for generic $\lambda_1,\hdots,\lambda_n \in \mathbb{R}$) yields the atoms and the weights of the measure $\mu := \sum_{i=1}^r \omega_i \delta_{\xi_i}$. Precisely, the atoms are $\xi_i := ||P_{i}||_2^{-2}(P_i^* M_{k} P_i)_{1\leqslant k \leqslant n}$ where $P_i$ denotes the $i\textsuperscript{th}$ column of $P$ and the weights are \begin{equation} w_i := \frac{e_1 H_{0} V_r P_i}{(\xi_i^\alpha)_{\|\alpha\|_1\leqslant d_2-1} V_r P_i}. \end{equation}
\end{enumerate}
We apply the above procedure to retrieve a measure from the output of the semidefinite optimization in super-resolution.

\subsubsection{Toeplitz Prony}
\label{subsub:Toeplitz Prony}
We now describe a new version of Prony's method, 
which we refer to as ``Toeplitz Prony''. In the setting of polynomial
interpolation, Prony's method can be adapted to exploit the fact that
we are interested in finding an atomic measure supported on the torus
with real weights. As a result, fewer evaluations are necessary.
For simplicity, we described this idea in the univariate setting,
\BM{which is well-known in signal processing. We will describe and exploit
a multivariate extension, which also requires fewer evaluations}.

{Following \cite{kunis_multivariate_2016} we are searching for a measure of finite support of the form $\mu = \sum_{k=1}^r \omega_k \delta_{\xi_{k}}$ where the weights $\omega_k$ are real and the support points $\xi_k$ with coordinates of norm $1$.
  Prony's method is based on the fact that the polynomial $p(x) = x^r - \sum_{k=1}^{r-1} p_k x^k := (x-\xi_{1})\hdots (x-\xi_{r})$ satisfies
  $\int_{\mathbb{C}} q(z) p(z) d\mu(z) = 0$ for any $q\in \C[x]$. We consider instead the following relations
\begin{equation}
\begin{array}{rcl}
\int_{\mathbb{C}} \bar{z}^0 p(z) d\mu(z) &=& 0, \\
& \vdots & \\
\int_{\mathbb{C}} \bar{z}^{r-1} p(z) d\mu(z) &=& 0, \\
\end{array}
\end{equation}
yielding
\begin{equation}
\begin{pmatrix}
\sigma_0 & \hdots & \sigma_{r-1} \\
\vdots & & \vdots \\
\overline{\sigma_{r-1}} & \hdots & \sigma_{0}
\end{pmatrix}
\begin{pmatrix}
p_0 \\
\vdots \\
p_{r-1}
\end{pmatrix}
=
\begin{pmatrix}
\sigma_r \\
\vdots \\
\sigma_{1}
\end{pmatrix}
\end{equation}
where $\sigma_{k}= \int z^{k} d\mu$ and  $\overline{\sigma_{k}}= \int \overline{z}^{k} d\mu=  \int {z}^{-k} d\mu$ since $\mu$ has real weights and the coordinates of its support points are of norm $1$.
Note that only $r+1$ evaluations are needed and that the above matrix is a Toeplitz matrix, as opposed to the Hankel matrix of the Prony method. Both matrices have the same size, but to construct the Hankel matrix, $2r$ moments $\sigma_{k}$ are needed.
}

The approach can be extended to the multivariate case, with Toeplitz like moment matrices. The rows are indexed by monomials and columns indexed by anti-monomials, that is, monomials with negative exponents. The entries of the matrix  indexed by $(\alpha,-\beta)$ with $\alpha, \beta\in \N^{n}$ is $\sigma_{\alpha-\beta}=\int z^{\alpha-\beta}d\mu$.  The same algorithm as in the Hankel Prony approach can then be used to obtain the decomposition of the measure from its moments. Note that the variant of Prony's method  \cite{sauer_pronys_2016-1} 
(which also uses Toeplitz matrices) is computationally more demanding and thus not relevant here.

\subsubsection{Advanced Prony}
We now describe a more elaborate form of Prony's method, which we will refer to as ``Advanced Prony''.
The multivariate Prony method decomposes a multi-index sequence $\sigma=(\sigma_{\alpha})_{\alpha\in \N^{n}}\in \C^{\N^{n}}$, or equivalently a multivariate series, into a sum of polynomial-exponential sequences or series, from a finite set $\{\sigma_{\alpha}, \alpha\in A\subset \N^{n}\}$ of coefficients.

In the case of sparse interpolation, the coefficients
$\sigma_{\alpha}$ of the series are the values $g(\varphi^{\alpha})$ for
$\alpha\in A \subset \N^{n}$. If $g=\x^{\beta}$ $\beta\in \N^{n}$, the
corresponding series is the exponential series of $\xi$, where
$\xi = \varphi^{\beta}$.  Therefore if
$g= \sum_{i=1}^{r} \omega_{i} \x^{\beta_{i}}$ is a sparse polynomial,
the series $\sigma_{\alpha} = g(\varphi^{\alpha})$ decomposes into a sum
of $r$ exponential series with weights $\omega_{i}$ and frequencies
$\xi_{i}= \varphi^{\beta_{i}}$.  The weights $\omega_{i}$ are the
coefficients of the monomials of $g$ and the frequencies
$\varphi^{\beta_{i}}$ yield the exponents $\beta_{i}= \log_{\varphi}(\xi)$
of the monomials.

To compute this decomposition, we apply the following method.
Subsets of monomials $A_{0}, A_{1}\subset \x^{\N^{n}}$ are chosen adequately so that the rank of the Hankel matrix 
$$ 
H_{0}=(\sigma_{\alpha_{0}+\alpha_{1}})_{\alpha_{0}\in A_{0}, \alpha_{1}\in A_{1}}
$$
is the number of terms $r$. The Hankel matrices $H_{i}=(\sigma_{e_{i}+\alpha_{0}+\alpha_{1}})_{\alpha_{0}\in A_{0}, \alpha_{1}\in A_{1}}$ are also computed for $i=1,\ldots,n$ and $(e_{i})$ is canonical basis of $\N^{n}$.  The subsets $A_{0}, A_{1}$ are chosen so that the monomial sets $\x^{A_{0}}$ and $\x^{A_{1}}$ contain a basis of the quotient algebra of the polynomials modulo the vanishing ideal of the points $\xi_{1}, \ldots, \xi_{r}$.

Using Singular Value Decomposition \cite{harmouch_structured_2017} or a Gramm-Schmidt orthogonalization  process \cite{mourrain_polynomial-exponential_2016}, tables of multiplication by the variables in a basis of the associated Artinian Gorenstein algebra $\mathcal{A}_{\sigma}$ are deduced.  The frequencies $\xi_{i}\in \C^{n}$, which are the points of the algebraic variety associated to $\mathcal{A}_{\sigma}$, are obtained by solving techniques from multiplication tables, based on eigenvector computation. The weights $\omega_{i}$ can then be deduced from the eigenvectors of these multiplication operators.

To compute this decomposition, only the evaluations $g(\varphi^{\alpha})$ with $\alpha\in A=\cup_{i=1}^{n} e_{i}+A_{0}+A_{1}$ are required.

Naturally, the ``Advanced Prony'' can be adapted to the Hankel and Toeplitz cases described in the two previous sections, yielding approaches which we will refer to as ``Advanced H. Prony'' and ``Advanced T. Prony''.

\section{Sparse Interpolation}

In  \S \ref{Prony} we have seen how to solve the sparse interpolation problem via Prony's method. We now
consider two other approaches which both solve some convex optimization problem with
a sparsity-inducing criterion.

\subsection{A sparse recovery approach to interpolation}

Suppose that  $g^*\in\R[\x]_d$ is an unknown polynomial of degree $d$ and we can make a certain number of 
``black-box" evaluations  $g^*(\zeta_k)=\sigma_k$ at some points $(\zeta_k)\subset S$, $k=1,\ldots,s$, that we may choose to our convenience.
Consider the following optimization problem $\P$:
\begin{eqnarray}
\label{a1}
\P:& \rho=&\displaystyle\inf_{g\in\R[\x]_d}\,\{\,\Vert  g\Vert_1:\quad g(\zeta_k)\,=\,\sigma_k,\quad k=1,\ldots,s\,\}\\
\label{a2}
& =&\displaystyle\inf_{g\in\R[\x]_d}\,\{\,\Vert g\Vert_1:\quad \langle g,\delta_{\zeta_k}\rangle \,=\,\sigma_k,\quad k=1,\ldots,s\,\}
\end{eqnarray}
where $\delta_{\x_k}$ is the Dirac at the point $\zeta_k\in\R^n$, and $\langle\cdot,\cdot\rangle$
the duality bracket $\int_S fd\mu$ between $\mathscr{C}(S)$ and $\M(S)$. Equivalently $\P$ also reads:

\begin{equation}
\label{a10}
 \rho=\displaystyle\inf_{\,g}\,\{\displaystyle\sum_{\alpha\in\N^n_d}\vert g_\alpha\vert:\: \displaystyle
\sum_\alpha g_\alpha\,\zeta_k^\alpha\,=\,\sigma_k,\quad k=1,\ldots,s\,\},
\end{equation}
or in the form of an LP as:

\begin{equation}
\label{a11}
 \rho=\displaystyle\inf_{\,g^+,g^-\geq0}\,\{\displaystyle\sum_{\alpha\in\N^n_d}(g^+_\alpha+g^-_\alpha):\: \displaystyle
\sum_\alpha g^+_\alpha\,\zeta_k^\alpha-g^-_\alpha \zeta_k^\alpha\,=\,\sigma_k,\quad k=1,\ldots,s\,\}
\end{equation}
which is an LP. Let $\sigma=(\sigma_k)$, $k=1,\ldots,s$.
The dual of the LP (\ref{a11}) is the LP:

\begin{eqnarray}
\nonumber
\P^*:&\rho=\displaystyle\sup_{\lambda\in\R^s}&\{\,\sigma^T\lambda:\quad \vert\sum_{k=1}^s\,\lambda_k\,\zeta^\alpha_k\,\vert\leq\,1;\quad\alpha\in\N^n_d\,\}\\
\nonumber&=\displaystyle\sup_{\lambda\in\R^s}&\{\sigma^T\lambda:\quad 
\underbrace{\vert\langle\x^\alpha,\displaystyle\sum_{k=1}^s\lambda_k\,\delta_{\zeta_k}\rangle\vert}
_{=\vert m_{\lambda}(\alpha)\vert}\,\leq\,1;\quad\alpha\in\N^n_d\,\}\\
\label{a5}
&=\displaystyle\sup_{\lambda\in\R^s}&\{\sigma^T\lambda:\quad \Vert m_{\lambda}\Vert_\infty\,\leq\,1\,\},
\end{eqnarray}
where to every $\lambda\in\R^s$ is associated the vector $m_\lambda\in\R^{s(d)}$ defined by
\[m_\lambda(\alpha)\,:=\,\langle \x^\alpha, \sum_{k=1}^s \lambda_k\,\delta_{\zeta_k}\rangle\,=\,\sum_{k=1}^s\lambda_k\,\zeta_k^\alpha .\]
So in the dual $\P^*$ one searches for $\lambda^*$, equivalently the signed atomic measure $\mu^*:=
\sum_{k=1}^s\lambda^*_k\,\delta_{\zeta_k}$, as we also do in super-resolution (\ref{super-resolution}) (but in $\P^*$ the support is known).

\begin{lem}\label{lem:3.1}
Let $(\hat{g}^+,\hat{g}^-)$ be an optimal solution of the naive LP (\ref{a11}) with associated polynomial
$\x\mapsto \hat{g}(\x):=\hat{g}^+(\x)-\hat{g}^-(\x)$ \BM{and $s=r$
points of evaluation $\zeta_{1}, \ldots,\zeta_{r}$}. Let $\lambda^*\in\R^s$ be an optimal solution of its dual (\ref{a5}). Then:

(i) $\hat{g}$ has at most $r$ non-zero entries, out of potentially $s(d)={n+d\choose n}$. 

(ii) $-1\leq m_{\lambda^*}(\alpha)\leq 1$ for all $\alpha\in\N^n_d$, and 
\begin{equation}
\label{complementarity}
\hat{g}_\alpha\,>\,0\quad\Rightarrow\quad \underbrace{m_{\lambda}(\alpha)=1}_{\int\x^\alpha d\mu^*=1};\qquad
\hat{g}_\alpha\,<\,0\quad\Rightarrow\quad \underbrace{m_{\lambda}(\alpha)=-1}_{\int \x^\alpha d\mu^*=-1},\end{equation}
\end{lem}
\BM{\begin{proof}
By standard arguments in Linear Programming, an optimal solution of
(\ref{a11}) is a vertex of the associated polytope, with at most $r$
non-zero entries. This proves ($i$).

To prove ($ii$), we check that
$\hat{g}_\alpha^+ \times \hat{g}_\alpha^-=0$ at the optimal solution
$\hat{g}$, since the columns of $A$ associated to $g^{+}_{\alpha}$ and
$g^{-}_{\alpha}$ are opposite and the basis columns defining the
vertex $\hat{g}$ are independent.  Let $\hat{g}_\alpha=
\hat{g}_\alpha^+ - \hat{g}_\alpha^-$. Then, the
complementary slackness condition (\cite{DantzigLinearprogrammingIntroduction1997}[Theorem 5.4]) implies that if
$\hat{g}_\alpha = \hat{g}_\alpha^+>0$ then $m_{\lambda}(\alpha)=1$,
and if $\hat{g}_\alpha = - \hat{g}_\alpha^- <0$ then
$m_{\lambda}(\alpha)=-1$.
\end{proof}}
 
\noindent
{\bf Exact recovery.}
\BM{Lemma \ref{lem:3.1} shows that an optimal solution $\hat{g}$ of
  (\ref{a11}) corresponds to a sparse polynomial with at most $r$ non-zero terms. But it may not
  coincide with the sparse polynomial $g$.}
A natural issue is {\it exact recovery} \BM{by increasing the number $s$ of sampling}, i.e., is there a value of $s$ (with possibly $s\ll O(n^d)$) for which $\hat{g}=g^*$? And if yes, how small $s$ must be?

A well-known and famous condition for exact recovery of sparse solution $\x^*$ to 
\begin{equation}
\label{sparse}
\min_\x \{\,\Vert \x\Vert_1:\quad \A\,\x\,=\,b\,\}
\end{equation}
is the so-called {\it Restricted Isometry Property} (RIP) of the matrix $\A$ introduced in Cand\`es and Tao
\cite{candes-tao} (see also Cand\`es \cite[Definition 1.1]{candes})
from which celebrated results of Cand\`es et al. \cite{candes-romberg} in compressed sensing could be obtained.

The interested reader is also referred to Fan and Kamath \cite{fan-kamath} for an interesting recent comparison of 
various algorithms to solve (\ref{sparse}) (even in the case where the RIP does not hold). 

It turns out
that for Problem (\ref{a10}) the resulting (Vandermonde-like) matrix $\A$ does {\it not}
satisfy the RIP property, and so exact recovery of the sparse polynomial as solution of (\ref{a10})
(and equivalently of the LP (\ref{a11})) \BM{cannot be guaranteed by
these techniques.}

\subsection{A formal analogy with super-resolution}
\label{analogy}
Observe that (\ref{a2}) is the analogue in function spaces of the {\it super-resolution} problem in measure spaces.
Indeed in both dual problems (\ref{super-resolution-dual}) and (\ref{a5}) one searches for a real vector $\lambda\in\R^m$.
In the former it is used to build up a {\it polynomial} $h:=\sum_{k=1}^m\lambda^*_kg_k$ uniformly bounded by $1$ 
on $S$ ($\Vert h\Vert_\infty\leq 1$) 
 while in the latter it is used to form an {\it atomic measure} $\sum_{k=1}^m\lambda^*_k\delta_{\zeta_k}$ 
whose moments (up to some order $d$) are uniformly bounded by $1$ ($\Vert m_\lambda\Vert_\infty\leq 1$).

Moreover, let $(\mu^*,\lambda^*)$ be a pair of optimal solutions to
(\ref{super-resolution})-(\ref{super-resolution-dual}). Then $\mu^*=\mu^+-\mu^-$
where $\mu^+$ and $\mu^-$ are positive atomic measures respectively supported on disjoint sets
$X_1,X_2\subset S$, and each point
of $X_1$ (resp. $X_2$) is a zero of the polynomial $\x\mapsto 1-\sum_{k=1}^m \lambda^*_k\,g_k(\x)$
(resp. $\x\mapsto 1+\sum_{k=1}^m \lambda^*_k\,g_k(\x)$). That is:

\begin{equation}
\label{b1}
\sup_{\zeta\in S}\,\vert  \displaystyle\sum_{k=1}^m\lambda^*_k\,g_k(\zeta)\vert\,\leq\,1,\quad\mbox{ and }
\end{equation}
\begin{equation}
\label{b2}
\zeta\in\mathrm{supp}(\mu^+)\:\Rightarrow\:\sum_{k=1}^m  \lambda^*_k\,g_k(\zeta)\,=\,1;\quad
\zeta\in\mathrm{supp}(\mu^-)\:\Rightarrow\:\sum_{k=1}^m \lambda^*_k\,g_k(\zeta)\,=\,-1,
\end{equation}
(compare with (\ref{complementarity})).

\subsection{Sparse interpolation as super-resolution}

In \S \ref{analogy} we have shown that the ``sparse recovery" 
formulation (\ref{a1}) of the sparse interpolation problem
could be viewed as a ``formal analogue" in function spaces of the super-resolution problem in measure spaces.

In this section we show that sparse interpolation is in fact a true (as opposed to formal) super-resolution 
problem on the torus $\T^n\subset\C^n$, provided that evaluations are made at points chosen in a certain adequate manner. So let 
\[\x\mapsto g(\x)\,=\,\sum_{\beta\in\N^n_d}g_\beta\,\x^\beta,\]
be the black-box polynomial with unknown real coefficients $(g_\beta)\subset\R$. \\

\noindent
{\bf A crucial observation.} \BM{This simple observation, which is the key
  point in most of the sparse interpolation methods, consists to see
  evaluations of the black-box polynomial at well-chosen points as moments of an atomic-measure.}
Let $\varphi\in\T^n$ (with $\varphi\neq(1,\ldots,1)$) be fixed, e.g., of the form:
\begin{equation}
\label{zedezero}
\varphi\,:=\,(\exp{(2i\pi/N_{1})},\ldots,\exp{(2i\pi/N_{n})}),\end{equation}
for some arbitrary (fixed) \BM{non-zero integers $N_{i}$}, or
\begin{equation}
\label{zedezero-1}
\varphi\,:=\,(\exp{(2i\pi\theta_1)},\ldots,\exp{(2i\pi\theta_1)}),\end{equation}
for some arbitrary (fixed) \BM{$\theta_i\in\R\setminus \N$}. 
With the choice (\ref{zedezero-1}) 
\[\mbox{[ $\alpha,\beta\in\mathbb{Z}^n$  and $\alpha\neq\beta$ ]}\quad\Rightarrow\quad
\varphi^\alpha\,\neq\,\varphi^\beta.\]
whereas with the choice (\ref{zedezero}) 
\[\mbox{[ $\alpha,\beta\in\mathbb{Z}^n$, $\max_i\max[\vert\alpha_i\vert,\vert\beta_i\vert]<N$,  and $\alpha\neq\beta$ ]}\quad\Rightarrow\quad
\varphi^\alpha\,\neq\,\varphi^\beta.\]
Next for every $\alpha\in\N^n$ :
\begin{eqnarray}
\label{values}
\sigma_\alpha\,:=\,g(\varphi^\alpha)\,=\,\sum_{\beta\in\N^n}g_\beta\,(\varphi^\alpha)^\beta&=&
\sum_{\beta\in\N^n}g_\beta\,(\varphi_{1}^{\alpha_1})^{\beta_1}\cdots (\varphi_{n}^{\alpha_n})^{\beta_n}\\
\nonumber
&=&\sum_{\beta\in\N^n}g_\beta\,(\varphi_{1}^{\beta_1})^{\alpha_1}\cdots (\varphi_{n}^{\beta_n})^{\alpha_n}\\
\nonumber
\label{equiv}
&=&\sum_{\beta\in\N^n}g_\beta\,(\varphi^\beta)^\alpha\,=\,\int_{\T^n}\z^\alpha\,d\mu_{g,\varphi}(\z),
\end{eqnarray}
where $\mu_{g,\varphi}$ is the signed atomic-measure on $\T^n$ defined by:
\begin{equation}
\label{equiv-measure}
\mu_{g,\varphi}\,:=\,\sum_{\beta\in\N^n}g_\beta\,\delta_{\xi_\beta}\quad\mbox{(and $\Vert\mu_{g,\varphi}\Vert_{TV}=\sum_\beta \vert g_\beta\vert=\Vert g\Vert_1$)},
\end{equation}
where $\xi_\beta:=\varphi^\beta\in\T^n$, for all $\beta\in\N^n$ such that $g_\beta\neq0$, and $\delta_\xi$
is the Dirac probability measure at the point $\xi\in\C^n$.\\

\begin{center}
In other words: {\it Evaluating $g$ at the point $\varphi^\alpha\in\T^n$ is the same as evaluating the moment $\int_{\T^n}\z^\alpha\,d\mu_{g,\varphi}$
of the signed atomic-measure $\mu_{g,\varphi}$. Therefore, the sparse interpolation problem is the same as recovering the finitely many unknown weights $(g_\beta)\subset\R$ and supports $(\varphi^\beta)\subset\T^n$ of the signed measure $\mu_{g,\varphi}$ on $\T^n$, from 
finitely many $s$ moments of $\mu_{g,\varphi}$, that is, a super-resolution problem.}
\end{center}
\begin{rem}
\label{remark-1}
The $n$-dimensional torus $\T^n$ is one among possible choices but any other choice of a set 
$S\subset\C^n$ and $\varphi\in\C^n$  (or $S\subset\R^n$ and $\varphi\in\R^n$) 
is valid provided that $(\varphi^\alpha)_{\alpha\in\N^n}\subset S$.
For instance $\varphi\in (-1,1)^{n}$ and $S:=[-1,1]^n$ is another possible choice. 
\BM{As the maximal degree of the powers $\varphi^{\alpha}$ required to reconstruct a sparse
  polynomial $g$ with $r$ non-zero terms is rapidly decreasing with
  the dimension $n$, choosing $S:=[-1,1]^n$ is also reasonable from a
  numerical point of view when $n>1$. This claim is corroborated by the
  numerical experiments in Section \ref{sec:experiments}. 
}
\end{rem}

Let $\Ascr^{1}_{\fc} = \{\alpha \in \Z^{n}\mid \Vert \alpha \Vert_{1}=
\sum_{i} \vert \alpha_{i}\vert \le \fc\}$.
With the choice $S:=\T^n$, $\varphi\in\T^n$ as  in (\ref{zedezero}) or in (\ref{zedezero-1}), and $\fc\in\N$, consider  the optimization problems:

\begin{equation}
\label{a3}
\rho_{\fc}\,=\,\displaystyle\inf_{ \mu\in\M(\T^n)}\,\{\,\Vert \mu\Vert_{TV}:\quad \int_{\T^n}\z^\alpha\,d\mu(\z) \,=\,\sigma_\alpha,\quad \alpha\in\Ascr^{1}_{\fc}\},
\end{equation}
where $\sigma_\alpha=g(\varphi^\alpha)$ is obtained from the black-box polynomial $g\in\R[\z]$, and
\begin{equation}
\label{a3-dual}
    \rho^*_{\fc}\,=\,\displaystyle\sup_{g\in\C[\z;\mathscr{A}_{\fc}]}\,\{\,\Re(\sigma^Tg) :\Vert \Re(g(\z))\Vert_\infty \leq 1\}
\end{equation}
(where $\a=(\sigma_\alpha)$).
Notice that the super-resolution problem (\ref{a3}) has the following 
equivalent formulation in terms of an infinite dimensional LP
\begin{equation}
\label{a33}
\rho_{\fc}\,=\,\displaystyle\inf_{\mu^+,\mu^-\in\M(\T^n)}\,\{\,\int_{\T^n}1\,d(\mu^++\mu^-):
\int_{\T^n}\z^\alpha\,d(\mu^+-\mu^-) 
\,=\,\sigma_\alpha,\quad \alpha\in\mathscr{A}_{\fc}\},
\end{equation}
with same dual (\ref{a3-dual}) as (\ref{a3}). Moreover $\rho_{\fc}=\rho^*_{\fc}$; the proof 
for $S=\T^n$ is very similar to the proof in De Castro et al. \cite{de_castro_exact_2017} for the case where $S\subset\R^n$ is a compact semi-algebraic set.

We next prove that the minimization problem \eqref{a3} has a
unique solution, provided that $\fc$ is sufficiently large.

For $d\in \NN$, let $\AA_{d}$ be the vector space spanned
by the monomials $\zb^{\alpha}$ with $\alpha\in \Ascr_{d}^{1}$.
For $\Xi=\{\xi_{1},\ldots,\xi_{r}\} \subset \T^{n}$, we denote 
by $\Ec(\Xi)$, the lowest $M=\max \{ |u_{i}(\xi_{r})|^{2} \}$ for all the families of interpolation polynomials $u_{1},\ldots,u_{r}\in \C[\zb^{\pm 1}]$ of
total degree $\le r-1$. We call $\Ec(\Xi)$ the {\em interpolation entropy}
of $\Xi$.
By standard arguments on the quotient algebra by an ideal defining $r$ points,
it is always possible to find a family of interpolation polynomials of
total degree $\le r-1$. 
Notice that $\Ec(\Xi)$ is related to the condition number of the Vandermonde matrix of the monomial basis of $\AA_{r-1}$ at the points $\Xi$. Thus it depends on the separation of these points.

To prove the unicity of the solution of the minimization problem \eqref{pb:mintv}, we first prove the existence of a dual polynomial. 
\begin{lem}\label{lm:dualpol}
Let $\Xi=\{\xi_{1},\ldots,\xi_{r}\} \subset \T^{n}$ and $\epsilon_{i}\in  \{\pm 1\}$.
Let $\fc \ge 4\, \pi\, r (r-1)\, \Ec(\Xi)$.
There exists $q(\zb)\in \AA_{\fc}$ such that
\begin{itemize}
  \item $q(\xi_{i}) = \epsilon_{i}$ for $i=1,\ldots,r$,
  \item $|q(\zb)|<1$ for $\zb$ in an open dense subset of $\T^{n}$.
\end{itemize}
\end{lem}
\begin{proof}
  Let $u_{i}(\zb)\in \CC[\zb^{\pm 1}]$, $i=1,\ldots, r$ be a family of interpolation polynomials at $\Xi$, with support in $\AA_{r-1}$ and which reaches $\Ec(\Xi)$. They satisfy the following properties: $U_{i}(\xi_{j})=\delta_{i,j}$ for $1 \le i,j \le r$.
We denote $M=\max_{\zb \in T^{n}}\{|u_{i}(\zb)|^{2}, i=1,\ldots,r\}$.

For $v_{1},\ldots, v_{r} \in \RR$, let 
\begin{equation}
U(\zb)= \sum_{i=1}^{r} v_{i} \, u_{i}(\zb)\overline{u}_{i}(\zb^{-1})\label{eq:polU}
\end{equation}
It is a polynomial with support in $\AA_{\fc}$ and with real values $\sum_{i=1}^{r} v_{i} \, |u_{i}(\zb)|^{2}$ for $\zb\in \T^{n}$.

A direct computation shows that
$$ 
\max_{\zb \in \T^{n}}|U(\zb)| \le r M m
$$
where $m=\max_{i}\{|v_{i}|\}$.
For $m\le {1 \over r M}$ and $\zb \in \T^{n}$, $|U(\zb)| \leq 1$.

Let us choose a Tchebychev polynomial $T(x)$ of degree $d$ big enough so that it has $2$ extremal points
$\zeta_{-}, \zeta_{+} \in ]-{1\over rM}, {1 \over rM}[$
with $T(\zeta_{-})= -1$ and $T(\zeta_{+})= 1$.
We can choose for instance $d$ such that ${\pi\over d}\le {1\over
  2 r M}$, that is $d \ge {2 \pi r M }$.
On the interval $[-1,1]$ outside the roots of $T'(x)=0$,
the norm $|T(x)|$ is strictly less than $1$.

Let $U(\zb)$ be the polynomial \eqref{eq:polU} constructed with $v_{i}=\zeta_{+}$ if $\epsilon_{i}=1$ and $v_{i}=\zeta_{-}$ if $\epsilon_{i}=-1$
and let $q(\zb)= T(U(\zb)) \in \CC[\zb^{\pm 1}]$. Since $u_{i}\in
\AA_{r-1}$ and
$d\ge 2 \pi r \Ec(\Xi)$, we check that $q(\zb)\in \AA_{\fc}$ for $\fc \ge 4 \pi r (r-1) \Ec(\Xi) $.

Then we have $q(\xi_{i})= T(\zeta_{\epsilon_{i}})=\epsilon_{i}$ and for $\zeta\in \T^{n}$, $|q(\zb)|\le 1$ since $|U(\zb)| \in [-1,1]$ and $|T(x)| \le 1$ for $x\in [-1,1]$. Moreover, for $\zb \in \T^{n}$, $|q(\zb)|= 1$ only when $U(\zb)$ reaches a root of $T'(x)$ on $[-1,1]$.
This cannot be the case on a dense open subset of $\T^{n}$, since $U(\zb)$ is a non-constant polynomial. Thus $|q(\zb)|<1$ for $\zb$ in a dense open subset of $\T^{n}$.
\end{proof}

We can now prove the unicity of the minimizer of \eqref{pb:mintv}, by
an argument similar to the one used in
\cite{candes_towards_2014}[Appendix A].

\begin{thm}\label{thm:unicity}
Let $\rho = \sum_{i} \omega_{i}\, \delta_{\xi_{i}}$ be a measure
supported on points $\Xi=\{\xi_{1}, \ldots, \xi_{r}\} \subset \T^{n}$
with $\omega_{i}\in \RR$. Let $\fc
\ge 
2\pi (r-1)\, r \Ec(\Xi)$ and let $g_{1}, \ldots, g_{m}$ be a basis of $\AA_{\fc}$ and
$a_{k}= \int_{\T^{n}} g_k\,d\rho$, $k=1,\ldots,m$. 
The optimal solution of
\begin{equation}
\label{pb:mintv}
  \displaystyle\inf_{\mu\in \M(\T^{n})}\,\{\,\Vert \mu\Vert_{TV}:\quad  \int_{\T^{n}} g_k\,d\mu \,=\,a_k,\quad k=1,\ldots,m\,\}
\end{equation}
is the measure $\rho$.
\end{thm}
\begin{proof}
Let $\rho^{*}$ be the optimal solution of \eqref{pb:mintv}. It 
can be decomposed as $\rho^{*}=\rho + \nu$. The Lebesgue decomposition of $\nu$ at $\rho$ (see \cite{RudinRealComplexAnalysis1986}[Theorem 6.9]) is of the form $\nu=\nu_{\Xi} + \nu_{c}$
where $\nu_{\Xi}$ is supported on $\Xi$ and $\nu_{c}$ is supported on $\T^{n}\setminus \Xi$.

By Radon-Nykodim Theorem \cite{RudinRealComplexAnalysis1986}[Theorem 6.9], $\nu_{\Xi}$ has a density function $h$ with respect to $\rho$.
Let $\epsilon_{i}= sign(h(\xi_{i}) \omega_{i}) \in \{\pm 1\}$,
$i=1,\ldots, r$ and let $q(\zb)\in \CC[\zb^{\pm 1}]$ be the polynomial constructed from $\Xi$ and $\epsilon_{1}, \ldots, \epsilon_{r}$ as in Lemma \ref{lm:dualpol}.
We have 
$$ 
\int_{\T^{n}} q(\zb) d\nu_{\Xi} 
= \sum_{i=1}^{r} \epsilon_{i} h(\xi_{i}) \omega_{i}
= \sum_{i=1}^{r} |h(\xi_{i})| |\omega_{i}| = ||\nu_{\Xi}||_{TV}.
$$
Since the moments of monomials in $A$ are the same for
$\rho$ and $\rho^{*}$, we have
\begin{equation} \label{eq:intzero}
0 = \int_{\T^{n}} q(\zb) d\nu = 
\int_{\T^{n}} q(\zb) d\nu_{\Xi}
+\int_{\T^{n}} q(\zb) d\nu_{c}
= ||\nu_{\Xi}||_{TV} +\int_{\T^{n}} q(\zb) d\nu_{c}
\end{equation}

Since $|q(\zb)|< 1$ on a dense open subset of $\T^{n}$, if $\nu_{c}\neq 0$ then
$|\int_{\T^{n}} q(\zb) d\nu_{c}|< ||\nu_{c}||_{TV}$ and $ ||\nu_{\Xi}||_{TV}< ||\nu_{c}||_{TV}$.

Assuming $\nu_{c}\neq 0$, we have 
\begin{eqnarray*}
\lefteqn{||\rho||_{TV}} \\
 & \ge & ||\rho^{*}||_{TV}= ||\rho+\nu||_{TV}  =  ||\rho+\nu_{\Xi}||_{TV} + ||\nu_{c}||_{TV}\\
&\ge& ||\rho||_{TV}- ||\nu_{\Xi}||_{TV} + ||\nu_{c}||_{TV} > ||\rho||_{TV}
\end{eqnarray*}
This is a contradiction. Thus $\nu_{c}=0$, which implies by \eqref{eq:intzero} that $||\nu_{\Xi}||_{TV}=0$ and that $\nu=0$ and $\rho^{*}=\rho$.
\end{proof}

We can now prove that the optimal solution of the super-resolution
problem \eqref{a3} yields the coefficients and exponents of the sparse
polynomial, provided enough moments are known.
\BMe

\begin{thm}\label{inter-super}
Let $g^*\in\R[\z]$, $\x\mapsto g^*(\z):=\sum_\beta g^*_\beta\z^\beta$, be an unknown real polynomial.
Let $\Gamma:=\{\beta\in\N^n: g^*_\beta\neq0\}$ and $r:=\vert\Gamma\vert$.
Let $\varphi\in\T^n$ be as in (\ref{zedezero-1}) or in (\ref{zedezero}) 
(in which case $N>\displaystyle\max_{i=1,\ldots,n}\max\{\beta_i:\beta\in\Gamma\}$),
and 
$\sigma_\alpha=g(\varphi^\alpha)$, $\alpha\in\mathscr{A}_{\fc}$.
Let $\delta = \Ec(\varphi^{\Gamma})=\Ec(\{\varphi^{\alpha}\mid \alpha \in \Gamma\})$.
There is a constant $C>0$ (that depends only on $r$)
such that if $\fc \ge C\, \delta$ then the optimization problem (\ref{a3}) has a unique optimal solution $\mu^*$ such that
\begin{equation}
\label{carac}
\mu^*\,:=\,\sum_{\beta\in \Gamma} g^*_\beta\,\delta_{\varphi^\beta}\quad\mbox{and}\quad\Vert\mu^*\Vert_{TV}=
\sum_\beta\vert g^*_\beta\vert.
\end{equation}
In addition, there is no duality gap (i.e., $\rho_{\fc}=\rho^*_{\fc}$),  (\ref{a3-dual}) has an optimal solution $g^*\in\C[\x;\mathscr{A}_{\fc}]$, and 
\begin{equation}
\label{complementarity-cond}
g^{*+}_\beta\,>\,0\,\Rightarrow \Re(g^*(\varphi^\beta))\,=1;\quad
g^{*-}_\beta\,>\,0\,\Rightarrow \Re(g^*(\varphi^\beta))\,=-1.
\end{equation}

\end{thm}
\proof
Of course the measure $\mu^*$ in (\ref{carac}) is feasible for (\ref{a3}). From the definition of $N$ and $\Gamma$,
all points $(\varphi^\alpha)\subset\T^n$, $\alpha\in\Gamma$, are distinct whenever $\varphi$ is chosen as in (\ref{zedezero}) or
in (\ref{zedezero-1}).

Moreover, by Theorem \ref{thm:unicity}, under the condition $\fc \ge 4 \pi r (r-1) \Ec(\varphi^{\Gamma})= C \Ec(\varphi^{\Gamma})$ with $C= 4 \pi r (r-1)$, the optimal solution of (\ref{a3}) is unique and 
is the sparse measure on $\T^n$ that satisfies the moment conditions of (\ref{a3}), i.e., $\mu^*$.

Next, write the optimal solution $\mu^*$ of (\ref{a3}) as $\mu^*=\mu^+-\mu^-$ for two signed Borel measures 
$\mu^+,\mu^-\in\mathscr{M}(\T_n)_+$, i.e.,
\[\mu^+=\sum_{\beta}g_\beta^{*+}\,\delta_{\varphi^\beta};\quad
\mu^-=\sum_{\beta}g_\beta^{*-}\,\delta_{\varphi^\beta}.\]
We have already mentioned that from
\cite{de_castro_exact_2017}, the optimal values of
(\ref{a3}), (\ref{a3-dual}) and (\ref{a33}) are the same,  i.e., $\rho_{\fc}=\rho_{\fc}^*$, and
therefore the measures $\mu^+$ and $\mu^-$ are optimal solutions of (\ref{a33}).
Let $f^*$ be an optimal solution of (\ref{a3-dual}). One relates
$\mu^*$ and $f^*$ has follows.  
As $\rho_{\fc}=\Vert\mu^*\Vert_{TV}=\rho^*_c=\Re(\a^T\f^*)$, 
\begin{eqnarray*}
\Vert\mu^*\Vert_{TV}&=&\int_{\T^n} d(\mu^++\mu^-)\\
&=&\int_{\T^n}\underbrace{\Re(1-f^*)}_{\geq0}\,d\mu^+
\int_{\T^n}\underbrace{\Re(1+f^*)}_{\geq0}\,d\mu^-+\underbrace{\int_{\T^n}\Re(f^*)\,d(\mu^+-\mu^-)}_{=\Re(\a^T\f^*)}
\end{eqnarray*}
it follows that $\mu^+$ (resp. $\mu^-$) is supported on the zeros of $\Re(1-f^*)$ (resp. $\Re(1+f^*)$) on $\T^n$. $\qed$
\endproof
Therefore to recover $r$ points one needs at most $(2r\,C_n+1)^n$ evaluations.

\subsection{A hierarchy of SDP relaxations for solving the super-resolution problem}
Recall that in the super-resolution model described in Cand\`es and Fernandez-Granda \cite{candes_towards_2014}, 
one has to make evaluations in the multivariate case at all points $\varphi^\alpha$ with
\begin{equation}
\alpha\in\,\mathscr{A}_{\fc}:= \mathscr{A}_{\fc}^{\infty}\,=\,\{-\fc,-(\fc-1),\ldots,-1,0,1,\ldots,(\fc-1),\fc\}^n = \{ \alpha ~|~ \|\alpha\|_{\infty} \leqslant \fc\} \subset\mathbb{Z}^n,
\end{equation}
where $\| \alpha \|_{\infty} := \max \{ |\alpha_1| , \hdots , |\alpha_n| \}$. This makes perfect sense in such applications as image reconstruction from measurements (typically 2-dimensional objects)
of signal processing. However, for polynomial interpolation $\vert\mathscr{A}_{\fc}\vert$ is rapidly prohibitive if
one consider polynomials of say $n=10$ variables.
Indeed, if $n=10$ then the first order semidefinite program of the hierarchy entails matrix variables of size $1,024 \times 1024$. Bear in mind that currently, semidefinite programming solvers are limited to matrices of size a few hundred. Thus it is not possible to compute \textit{even} the first order relaxation! 

We propose to reduce the computational burden by using the one-norm truncation, i.e. $\| \alpha \|_{1} := |\alpha_1| + \hdots + |\alpha_n|$, by making evaluations at all points $\varphi^\alpha$ with
\begin{equation}
\alpha \in\,\mathscr{A}^1_{\fc}\,:=\,\{\alpha-\beta~|~ \alpha,\beta \in \mathbb{N}^n, ~ \|\alpha\|_1 , \|\beta\|_1 \leqslant \fc \}.
\end{equation}

\begin{figure}[!h]
	\centering
	\includegraphics[width=.6\textwidth]{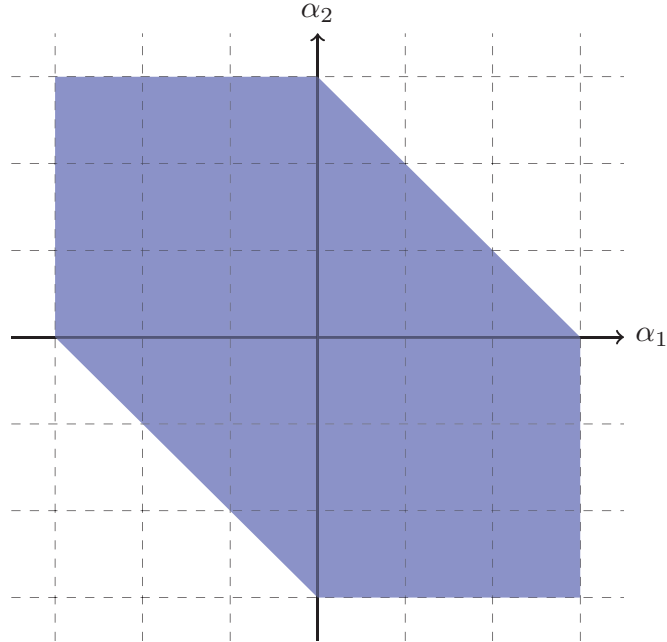}
	\caption{Evaluations at $\alpha-\beta$ with $|\alpha_1|+|\alpha_2| \leqslant 3$ and $|\beta_1|+|\beta_2| \leqslant 3$ and $\alpha_1,\alpha_2,\beta_1,\beta_2 \in \mathbb{N}$.}
\label{fig:eval1}
\end{figure}

An illustration is provided in Figure \ref{fig:eval1}. In addition:
\begin{equation}
\label{eq:tilde}
\forall \fc \in \mathbb{N}, ~~ \exists \tilde{\fc} \in \mathbb{N}:  ~~ \forall l \geqslant \tilde{\fc}, ~~~~ \mathscr{A}_{\fc} \subset \mathscr{A}^1_{l}.
\end{equation}
Thus, all the theoretical results of \cite{candes_towards_2014} are preserved. To appreciate the gain in using $\mathscr{A}^1_{\fc}$, for 10 variables, the first order semidefinite program of the hierarchy entails matrix variables of size $11 \times 11$ (instead of $1,024 \times 1,024$) and $56$ linear equalities. It is thus possible to compute the first order relaxation. The second order relaxation entails matrices of size $66 \times 66$ (instead of $59,049 \times 59,049$) and $1,596$ linear  equalities.

Notice also that Toeplitz Prony of Section \ref{subsub:Toeplitz Prony} uses $\mathscr{A}^1_{\fc}$ instead of $\mathscr{A}_{\fc}$ and is guaranteed to recover the optimal solution provided that $\fc$ is large enough.

As the super-resolution problem (\ref{a3}) is concerned with atomic measures finitely supported on the multi-dimensional torus
$\T^n$, we can adapt to the torus $\T^n$  the hierarchy of semidefinite programs
defined in De Castro et al. \cite{de_castro_exact_2017} for solving super-resolution problems with atomic measures 
(finitely supported) on semi-algebraic subsets of $\R^n$. 
For every fixed $\fc$, at step $d\geq \fc$ of the hierarchy,  the semidefinite program to solve reads:
\begin{equation}
\label{sdp-Toeplitz}
(\textbf{P}_{d,\fc}) \qquad  \qquad \qquad \qquad
\begin{array}{ll}
\rho_{d,\fc}\,=\displaystyle\inf_{y^+,y^-} & y^+_{0} + y^-_{0} \\[.5em]
\text{s.c.} & y^+_{\alpha} - y^-_{\alpha} = \sigma_{\alpha}~,~~~ \forall \alpha\in\mathscr{A}^1_{\fc}\\[.5em]
&\mathbf{T}_{d}(\y^+)\,\succeq\,0,\:\mathbf{T}_{d}(\y^-)\,\succeq\,0,
\end{array} \qquad \qquad \qquad \qquad
\end{equation}
where the Hermitian matrix $\mathbf{T}_{d}(y^+)$ has its rows and columns indexed in $\{ \alpha ~|~\|\alpha\|_1 \leqslant d \}$
and $\mathbf{T}_{d}(\y^+)_{\alpha,\beta}=y^+_{\beta-\alpha}$, for every $\|\alpha\|_1,\|\beta\|_1 \leqslant d$, and
similarly for the Hermitian matrix $\mathbf{T}_{d}(\y^-)$. In the univariate case $\mathbf{T}_{d}(\y^+)$
$\mathbf{T}_{d}(y^-)$ are Toeplitz matrices. When $\y^+$ is coming from a measure $\mu^+$ on $\T^n$
then $y^+_{\alpha}=\int_{\T^n}\z^{\alpha}\,d\mu^+(\z)$.
Clearly, (\ref{sdp-Toeplitz}) is a relaxation of (\ref{a33}) and so 
$\rho_d\leq \rho_{\fc}$ for all $d\geq \fc$. Moreover $\rho_d\leq\rho_{d+1}$ for all $d$.

Note that with the above notations, the ``Toeplitz Prony'' method proposed in Section \ref{subsub:Toeplitz Prony} consists in directly extracting a measure from the matrix $\mathbf{T}_{d}(a)$. In constrast, the super resolution approach consists of decomposing it into $\mathbf{T}_{d}(a) = \mathbf{T}_{d}(\y^+) - \mathbf{T}_{d}(\y^-)$, optimizing over $\y^+$ and $\y^-$, and then applying the ``Toeplitz Prony'' method to $\mathbf{T}_{d}(\y_*^+)$ and $\mathbf{T}_{d}(\y_*^-)$ at an optimal solution $(\y_*^+,\y_*^-)$ of \eqref{sdp-Toeplitz}.

\begin{lem}
For each $d\geq \fc$ the (complex) semidefinite program ($\textbf{P}_{d,\fc}$) in \eqref{sdp-Toeplitz} has an optimal solution $(\y^+,\y^-)$. 
In addition, if the rank conditions
\begin{eqnarray}
\label{rank+}
\text{rank} (\mathbf{T}_{d}(\y^+))&=&\text{rank} (\mathbf{T}_{d-2}(\y^+))\\
\label{rank-}
\text{rank} (\mathbf{T}_{d}(\y^-))&=&\text{rank} (\mathbf{T}_{d-2}(\y^-))
\end{eqnarray}
are satisfied then there exist two Borel atomic measures $\mu^+$ and $\mu^-$ on $\T^n$ such that:
\begin{equation}
\label{atoms}
y^+_{\alpha}\,=\,\int_{\T^n}\z^\alpha\,d\mu^+(\z)\quad\mbox{and}\quad
y^-_{\alpha}\,=\,\int_{\T^n}\z^\alpha\,d\mu^-(\z),\quad\forall\alpha\in\mathscr{A}^1_{d}.\end{equation}
The support of $\mu^+$ (resp. $\mu^-$) consists of ${\rm rank}(\mathbf{T}_{d}(\y^+))$ 
(resp. ${\rm rank}(\mathbf{T}_{d}(\y^-)))$  atoms on $\T^n$ which can be extracted by a numerical algebra routine (e.g. the Prony method described in Section \ref{subsub:Hankel Prony}).

In addition, if \eqref{rank+}-\eqref{rank-} hold for an optimal solution of ($\textbf{P}_{d,\tilde{\fc}}$) with $\tilde{\fc}$ as in \eqref{eq:tilde}, then under the separation conditions of Theorem \ref{inter-super}, the Borel measure 
$\mu^*:=\mu^+-\mu^-$ is the unique optimal solution of (\ref{a3}).
\end{lem}
\proof
Consider a minimizing sequence $(\y^{+,\ell},\y^{-,\ell})_{\ell\in\N}$ of (\ref{sdp-Toeplitz}). 
Since one minimizes $y^+_{0}+y^-_{0}$ one has
$y^{+,\ell}_{0}+y^{-,\ell}_{0}\leqslant y^{+,1}_{0}+y^{-,1}_{0}=:\rho$, for $\ell\geq1$.
The Toeplitz-like structure of $\mathbf{T}_{d}(\y^{+,\ell})$ and the psd constraint $\mathbf{T}_{d}(\y^{+,\ell})\succeq0$ imply
$\vert y^{+,\ell}_{\alpha}\vert \leqslant \rho$ for all $\alpha\in\mathscr{A}^1_{d}$; and similarly
$\vert y^{-,\ell}_{\alpha}\vert \leqslant \rho$ for all $\alpha\in\mathscr{A}^1_{\fc}$. Hence there is a subsequence $(\ell_k)$
and two vectors $\y^+=(y^+_\alpha)_{\alpha\in\mathscr{A}^1_{\fc}}$ and
$\y^-=(y^-_\alpha)_{\alpha\in\mathscr{A}^1_{\fc}}$,
such that
\[\lim_{k\to\infty}\y^{+,\ell_k}\,=\,\y^+
\quad\mbox{and}\quad\lim_{k\to\infty}\y^{-,\ell_k}\,=\,\y^-.\]
In addition, from the above convergence it also follows that
$(\y^+,\y^-)$ is a feasible solution of (\ref{sdp-Toeplitz}), hence an optimal solution
of (\ref{sdp-Toeplitz}).

Next, in the univariate case, a Borel measure $\mu^+$ and $\mu^-$ on $\T$ can always be extracted from the
semidefinite positive Toeplitz matrices $\mathbf{T}_{d}(\y^+)$ and $\mathbf{T}_{d}(\y^-)$ respectively.
This is true regardless of the rank conditions (\ref{rank+})-(\ref{rank-}) and was proved in \cite[p. 211]{iohvidov1982}.
In the multivariate case, $\mathbf{T}_{d}(\y^+)$ and
$\mathbf{T}_{d}(\y^-)$ are  Toeplitz-like matrices, and we may and will invoke the recent result \cite[Theorem 5.2]{cedric}. (Note that this is true for Toeplitz matrices, but {\it not} for general Hermitian matrices for which additional non-trivial conditions must be satisfied (see \cite[Theorem 5.1]{cedric}).)
It implies that a Borel measure $\mu^+$ (resp. $\mu^-$) on $\T^n$ can be extracted 
from a multivariate semidefinite positive Toeplitz-like matrix $\mathbf{T}_{d}(\y^+)$
(resp. $\mathbf{T}_{d}(\y^-)$) if the rank condition (\ref{rank+}) (resp. (\ref{rank-})) holds. Hence we have proved (\ref{atoms}).
 
Finally the last statement follows from Theorem \ref{inter-super} and the fact that (\ref{a3}) and (\ref{a33}) 
have the same optimal value and an optimal solution $(\mu^+,\mu^-)$ of (\ref{a33}) provides and optimal solution
$\mu^*=\mu^+-\mu^-$ of (\ref{a3}). $\Box$
\endproof

\noindent
{{\bf Asymptotics as $d$ increases.}
In case the conditions \eqref{rank+}-\eqref{rank-} do not hold, we still have the following asymptotic result at an optimal solution.

\begin{lem}
Assume that $\fc$ satisfies the conditions of Theorem \ref{inter-super} and let $\tilde{\fc}$ be as in \eqref{eq:tilde}. For each $d\geq \tilde{\fc}$, let $(\y^{+,d},\y^{-,d})$ be an optimal solution of (\ref{sdp-Toeplitz}).
Then for each $\alpha\in\Z^n$,
\begin{equation}
\label{asymptotics}
\lim_{d\to\infty}(y^{+,d}_{\alpha}-y^{-,d}_{\alpha})\,=\,\int_{\T^n}\z^\alpha\,d\mu^*,
\end{equation}
where the Borel signed measure $\mu^*$ on $\T^n$ is the unique optimal solution of (\ref{a3}) characterized in (\ref{carac}).
\end{lem}
\proof
As $\rho_{d,\tilde{\fc}}\leq\rho_{d+1,\tilde{\fc}}\leq \rho_{\fc}$ for all $d\geq \tilde{\fc}$ and $y^{+,d}_{0}+y^{-,d}_{0}=\rho_{d,\tilde{\fc}}$, it follows that
$\vert y^{+,d}_{\alpha}\vert\leq \rho_{\fc}$ and $\vert y^{-,d}_{\alpha}\vert\leq \rho_{\fc}$ for all $\alpha\in\mathscr{A}^1_{d}$
and all $d\geq \tilde{\fc}$. 
By completing with zeros, one may and will consider all finite-dimensional  vectors $\y^{+,d}$ and $\y^{-,d}$ as elements of a bounded set
of $\ell_\infty$. Next, by weak-$\star$ sequential compactness of the unit ball of $\ell_\infty$, there exist
infinite vectors $\y^+,\y^-\in\ell_\infty$, and a subsequence $(d_k)_{k\in\N}$ such that :
\begin{equation}
\label{weakstar}
\lim_{k\to\infty}\,y^{+,d_k}_{\alpha}\,=\,y^+_{\alpha};\quad \lim_{k\to\infty}\,y^{-,d_k}_{\alpha}\,=\,y^-_{\alpha},\quad \forall\alpha\in\Z^n.
\end{equation}
Moreover from the above convergence we also have $\mathbf{T}_d(\y^+)\succeq0$ and $\mathbf{T}_d(\y^-)\succeq0$ for all $d$. This in turn implies that
$\y^+$ (resp. $\y^-$) is the moment sequence of a Borel measure $\mu^+$ (resp. $\mu^-$) on $\T^n$. In addition,
the convergence (\ref{weakstar}) yields
\[\sigma_\alpha\,=\,\lim_{k\to\infty}(y^{+,d_k}_{\alpha}-y^{-,d_k}_{\alpha})\,=\,\int_{\T^n}\z^\alpha\,d(\mu^+-\mu^-),\quad\forall\alpha\in\mathscr{A}^1_{\tilde{\fc}},\]
and
\[\rho_{\fc}\geq\,\lim_{k\to\infty}\rho_{d_k,\fc}\,=\,\lim_{k\to\infty} (y^{+,d_k}_{0}+y^{-,d_k}_{0})\,=\,\int_{\T^n}d(\mu^++\mu^-)\,\geq\,\Vert \mu^+-\mu^-\Vert_{TV},\]
which proves that $(\mu^+,\mu^-)$ is an optimal solution of (\ref{a33}). Therefore $\mu^*:=\mu^+-\mu^-$ is an optimal solution
of (\ref{a3}) and thus unique when $\tilde{\fc}$ satisfies the condition of Theorem \ref{inter-super}. This also implies that 
the limit $y^+_{\alpha}$ (resp. $y^-_{\alpha}$) in (\ref{weakstar}) is the same for all converging subsequences $(d_k)_{k\in\N^n}$ and therefore, 
for each $\alpha\in\Z^n$, the whole sequence $(y^{+,d}_{\alpha})_{d\in\N}$ (resp. $(y^{-,d}_{\alpha})_{d\in\N}$) converges to 
$y^+_{\alpha}$ (resp. $y^-_{\alpha}$), which yields the desired result (\ref{asymptotics}). $\Box$
\endproof
\subsection{A rigorous sparse recovery LP approach}

In this section we take advantage of an important consequence of viewing sparse interpolation as
a super-resolution problem.  Indeed when $\varphi$ is chosen as in (\ref{zedezero}) 
we {\it know} that the (unique) optimal solution $\mu^*$ of (\ref{a3}) is supported on the  {\it a priori} fixed grid 
$(\exp(2i\pi k_1/N),\ldots,\exp(2i\pi k_n/N))$, where $0\leq k_i\leq N$, $i=1,\ldots,n$. That is,
(\ref{a3}) is a {\it discrete} super-resolution problem as described in Cand\`es and Fernandez-Granda \cite{candes_towards_2014}.
Therefore solving (\ref{a3}) is also equivalent to solving the LP: 
\[\min_\x\,\{\Vert \x\Vert_1:\:\A\,\x\,=\,\b\,\}\]
where $\x\in\R^{[0,1,\ldots,N]^n}$. The matrix $\A$
has its columns indexed by $\beta\in [0,1,\ldots,N]^n$ and its rows indexed by $\alpha\in\mathscr{A}_{\fc}$, while
$\b=(b_\alpha)_{\alpha\in\mathscr{A}_{\fc}}$ is the vector 
of black-box evaluations at the points $(\varphi^\alpha)$, $\alpha\in\mathscr{A}_{\fc}$. So
\begin{equation}
\label{Ab}
\A(\alpha,\beta)\,=\,(\varphi^\beta)^\alpha\,=\,\varphi_{1}^{\beta_1\alpha_1}\cdots \varphi_{n}^{\beta_n\alpha_n};
\quad b_\alpha\,=\,g(\varphi^\alpha),\end{equation}
for all $\alpha\in\mathscr{A}_{\fc}$ and $\beta\in[0,1,\ldots,N]^n$.

\begin{prop}
\label{prop1}
Under the conditions of Theorem \ref{inter-super}, the LP $\displaystyle\rho=\min_\x\,\{\Vert \x\Vert_1:\:\A\,\x\,=\,\b\,\}$ with
$\A$ and $\b$ as in (\ref{Ab}) has a unique optimal solution which is the vector
of coefficients of the polynomial $g^*$ of Theorem \ref{inter-super}.
\end{prop}
\proof
Let
$g\in\R[\x]$, $\z\mapsto g(\z):=\sum_\beta g_\beta\, \z^\beta$, be the polynomial with vector of coefficients 
\[\mathbf{g}=[g_\beta]_{\beta\in [0,\ldots,N]^n}.\]
Then by construction, $\g\in\R^{[0,\ldots,N]^n}$ is an admissible solution of the LP 
with $\A$ and $\b$ as in (\ref{Ab}). 
One has $g(\varphi^\alpha)=\sigma_\alpha=g^*(\varphi^\alpha)$ for all $\alpha\in\mathscr{A}_{\fc}$,
where $g^*$ is as in Theorem \ref{inter-super}. The Borel measures $\nu^+$ and $\nu^-$ on $\T^n$ defined by
\[\nu^+:=\sum_{\beta\in [0,\ldots,N]^n} \max[0,x_\beta]\,\delta_{\varphi^\beta};\quad
\nu^-:=\sum_{\beta\in [0,\ldots,N]^n} -\min[0,x_\beta]\,\delta_{\varphi^\beta},\]
are a feasible solution of (\ref{a33}) and the Borel signed measure $\nu:=\nu^+-\nu^-$
satisfies $\Vert\nu\Vert_{TV}=\Vert \nu^+\Vert+\Vert \nu^-\Vert$. Hence $\Vert\nu\Vert_{TV}\geq\Vert\mu^*\Vert_{TV}$
where $\mu^*$ is the optimal solution of (\ref{a3}).
So the optimal value $\rho$ of the LP satisfies $\rho\geq\Vert\mu^*\Vert_{TV}$. On the other hand
with $g^*$ as in Theorem \ref{inter-super}, let
\[\g^*=[g^{*}_\beta]_{\beta\in [0,\ldots,N]^n}.\]
Then $\Vert\g^*\Vert_1=\Vert\mu^*\Vert_{TV}\leq\rho$ and so $\Vert\g^*\Vert_1=\rho$, which proves that
$\g^*$ is an optimal solution of the LP. Uniqueness follows from the uniqueness of solution to
(\ref{a3}). $\qed$
\endproof

\section{Numerical experiments}\label{sec:experiments}

In the problem of polynomial interpolation, we are not given a number of evaluations to begin with, i.e. $\fc$. Rather, we seek to recover a blackbox polynomial using the least number of evaluations. Thus, one could set $\fc = 1$, then compute a hierarchy of SDPs of order $d=1,2,\hdots$. Next, set $\fc = 2$, and compute another hierarchy of order $d=2,3,\hdots$. This leads to a hierarchy of hierarchies, which is costly from a computational perspective. Thus, we propose a single hierarchy where we choose to make all possible evaluations at each relaxation order. Therefore we have fixed $d=\fc$ in (\ref{sdp-Toeplitz}) and let $\fc$ increase 
to see when we recover the desired optimal measure (polynomial $g^*$) of Theorem \ref{inter-super}.\footnote{In the univariate case, the optimal value of $(\textbf{P}_{d,\fc})$ in \eqref{sdp-Toeplitz} does not increase with $d$ when $d > \fc$. Indeed, for any optimal solution of $(\textbf{P}_{\fc,\fc})$, there exists a representing signed measure $\mu = \mu^+-\mu^-$ on the torus. However, one may not be able to extract this measure.}

In order to make a rigorous comparison with Prony's method, we use the same exact same procedure to extract the atomic measures from the output matrices of the semidefinite optimization as for Prony's method. For the super-resolution of order $d$, we use Prony with input measurements up to degree $2d$ (that way $d_1=d_2=d$ in Section \ref{subsub:Hankel Prony}) for each of the two Toeplitz matrices. In all numerical experiments, we use the threshold $\epsilon = 0.1$ for determining the rank of a matrix in its SVD decomposition. This threshold is also used to test the rank conditions \eqref{rank+}-\eqref{rank-}.

\subsection{Separation of the support}
Initially, super resolution was concerned with signal processing where the measurements are given and fixed and we have no influence on them. In constrast, in the super resolution formulation of an polynomial interpolation problem, we can choose where we make the measurements, that is the points where we want to evaluate the blackbox polynomial. This can have a strong influence on the seperation condition which guarantees exact recovery on the signal (our blackbox polynomial). We illustrate this phenomenon on the following example.
Suppose that we are looking for the blackbox polynomial
\begin{equation}
g(x) = 3 x^{20} + x^{75} - 6 x^{80}
\end{equation}
whose degree we assume to be less than or equal to 100. We consider such a high degree in order to well illustrate the notion of the separation of the support. Below, we will consider more realistic polynomials, limited to degree 10. We now investigate two different ways of making evaluations and their impact on the separation of the support, crucial for super-resolution.
Let us firstly evaluate the blackbox polynomial at the points 
\begin{equation}
(e^{2\pi i \over 101})^0, (e^{2\pi i \over 101})^1, (e^{2\pi i \over 101})^2, \hdots , (e^{2\pi i \over 101})^d \subset \mathbb{T}
\end{equation}
at step $d$ of the SDP hierarchy (i.e. $(\textbf{P}_{d,d})$ in \eqref{sdp-Toeplitz}).
The proximity of points on the torus is thus directly related to the proximity of the exponents of the polynomial. It can be seen in the left part of Figure \ref{fig:single} that some of the point on the torus are very close to one another. 

Let us secondly evaluate in the blackbox polynomial at the points 
\begin{equation}
(e^i)^0, (e^i)^1, (e^i)^2, \hdots , (e^i)^d \subset \mathbb{T}
\end{equation}
at step $d$ of the SDP hierarchy.
The proximity of points on the torus is thus no longer related to the proximity of the exponents of the polynomial.
It can be seen in the left part of Figure \ref{fig:multiple} that the
points on the torus are nicely spread out. This is not guaranteed to
be the case, but is expected to be true \BM{for small values of $d$
  (since $\mathrm{frac}(\frac{k}{2\,\pi})$ are well separated in
  $[0,1]$ for small $k\in \N$)}. In order to recover the blackbox polynomial once a candidate atomic measure is computed, we form a table of the integers $k=1,\hdots,d$ modulo $2\pi$. For each atom, we consider its argument and find the closest value in the table, yielding an integer $k$, i.e. the power of the monomial associated to the atom. The coefficient of the monomial is given by the weight of the atom.

We now provide numerical experiments. Table 1 and Table 2 show the optimal value and the number of atoms of the optimal measure $\mu$ at each order $d$. 
Graphical illustrations of the solutions appear in Figure \ref{fig:single} and Figure \ref{fig:multiple}. The dual polynomials in the right hand of the figures illustrate why a higher degree is needed when the points are closer.

\begin{table}[!htb]
    \begin{minipage}{.5\linewidth}
      \caption{Evaluation at roots of unity $e^{2 k \mathbf{i} \pi \over 101}$}
      \centering
 \begin{tabular}{|c|c|c|c|c|}
\hline
Order $d=\fc$ & $\|\mu\|_{TV}$ & $\# \text{supp}(\mu)$ \\
\hline
0 &  \hphantom{0}2.0000 & 1 \\
1 &  \hphantom{0}7.6618 & 2 \\
2 &  \hphantom{0}8.1253 & 3\\
3 & \hphantom{0}8.3655 & 5\\
4 &  \hphantom{0}8.7240 & 7\\
5 &  \hphantom{0}8.9882 & 9\\
6 &  \hphantom{0}9.3433 & 11\\
7 &  \hphantom{0}9.5837 & 13\\
8 &  \hphantom{0}9.7993 & 17\\
9 &  \hphantom{0}9.9436 & 19\\
10 &  \hphantom{0}9.9978 & 20\\
11& 10.0000 & 3\\
\hline
\end{tabular}
    \end{minipage}%
    \begin{minipage}{.5\linewidth}
      \centering
        \caption{Evaluation at $e^{k \mathbf{i} }$}
\begin{tabular}{|c|c|c|c|}
\hline
Order $d=\fc$ & $\|\mu\|_{TV}$ & $\# \text{supp}(\mu)$ \\
\hline
0 & \hphantom{0}2.0000 & 1 \\
1 & \hphantom{0}8.7759 & 2 \\
2 & \hphantom{0}9.2803 & 3 \\
3 & 10.0000 & 3 \\
\hline
\end{tabular}
    \end{minipage} 
\label{tab:svsm}
\end{table}
\begin{remark}
Before we move on to other examples, we note that naive LP with evaluations at random points on the real line requires about 50 evaluations on this example, compared with the 4 evaluations with super-resolution using multiple loops and in fact, the rigorous LP on the torus also requires 4 evaluations.
\end{remark}

\begin{figure}[!h]
\centering
\begin{minipage}{.5\textwidth}
  \centering
  \includegraphics[width=1.1\textwidth]{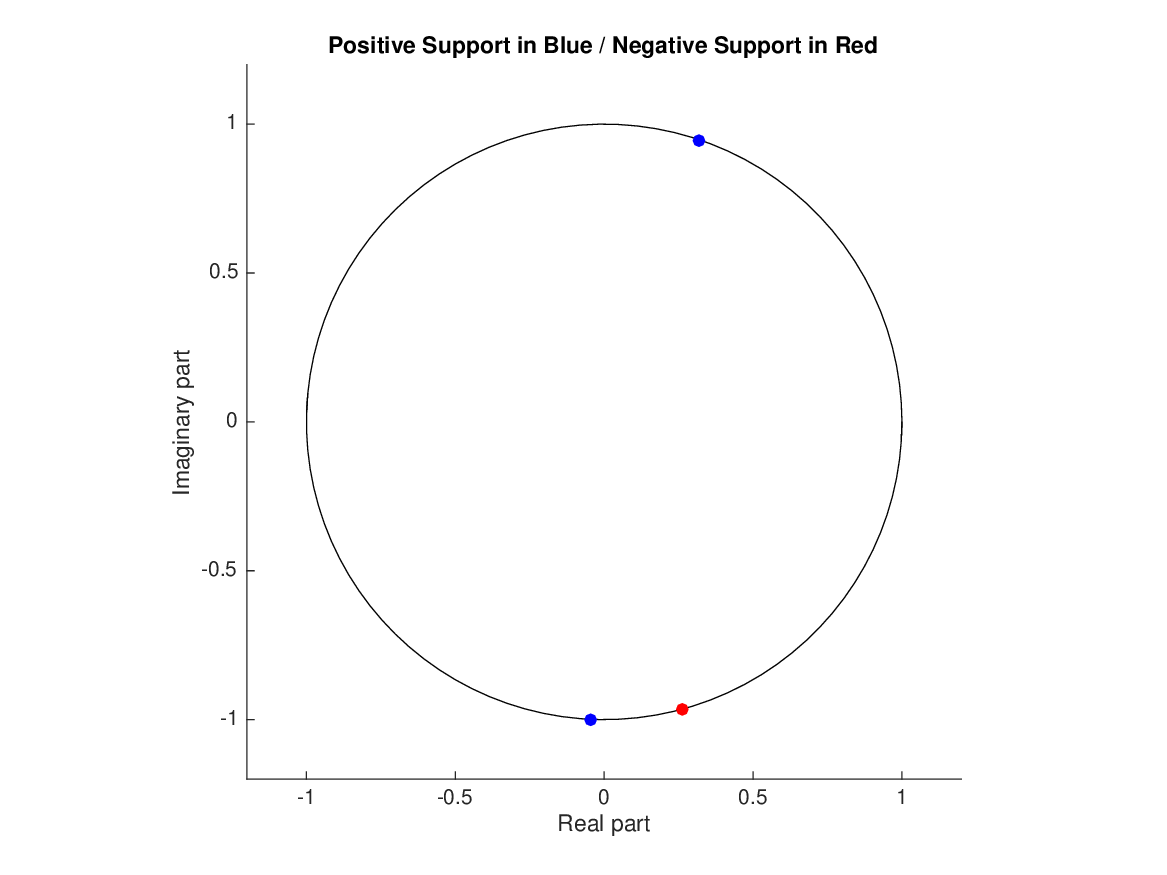}
  \label{fig:test1.1}
\end{minipage}%
\begin{minipage}{.5\textwidth}
  \centering
  \includegraphics[width=1.1\textwidth]{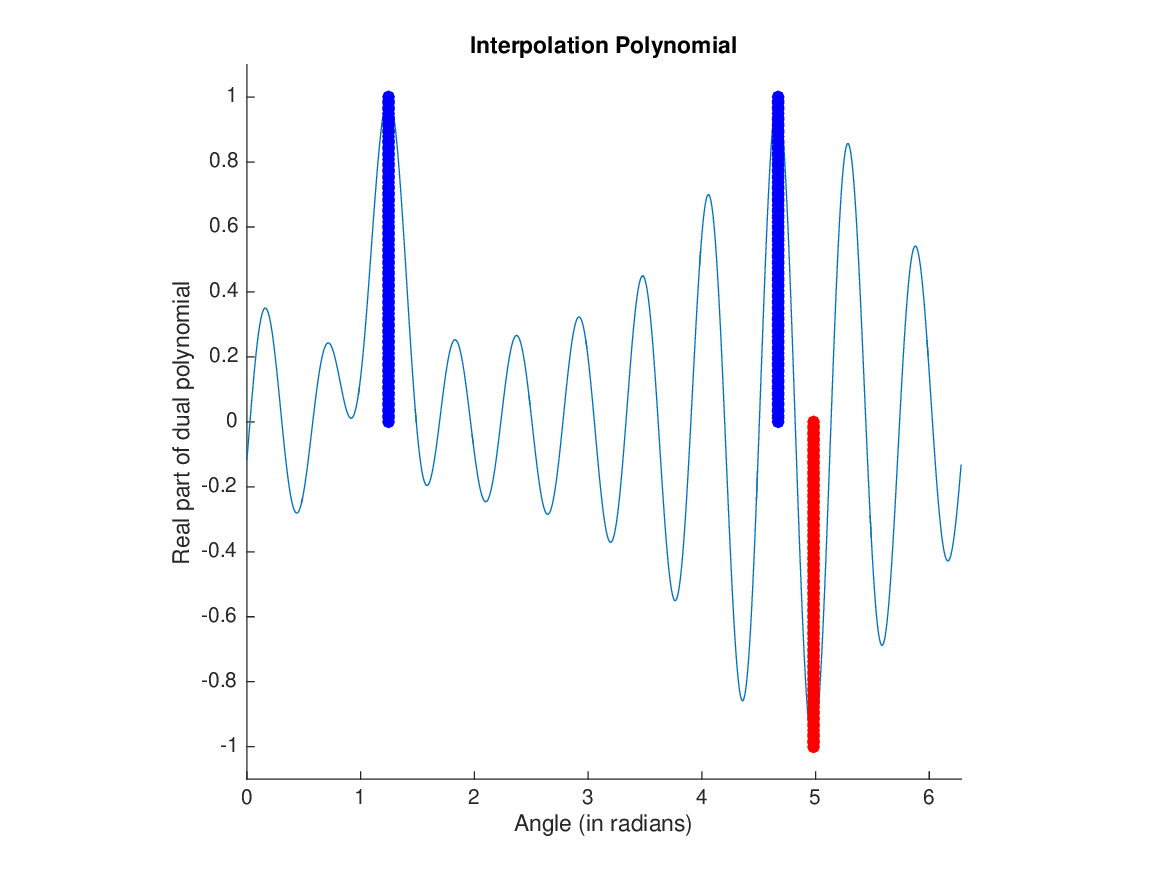}
  \label{fig:test1.2}
\end{minipage}
\caption{Primal-dual solution of super-resolution at order 11 (using single loop)}
\label{fig:single}
\end{figure}

\begin{figure}[!h]
\centering
\begin{minipage}{.5\textwidth}
  \centering
  \includegraphics[width=1.1\textwidth]{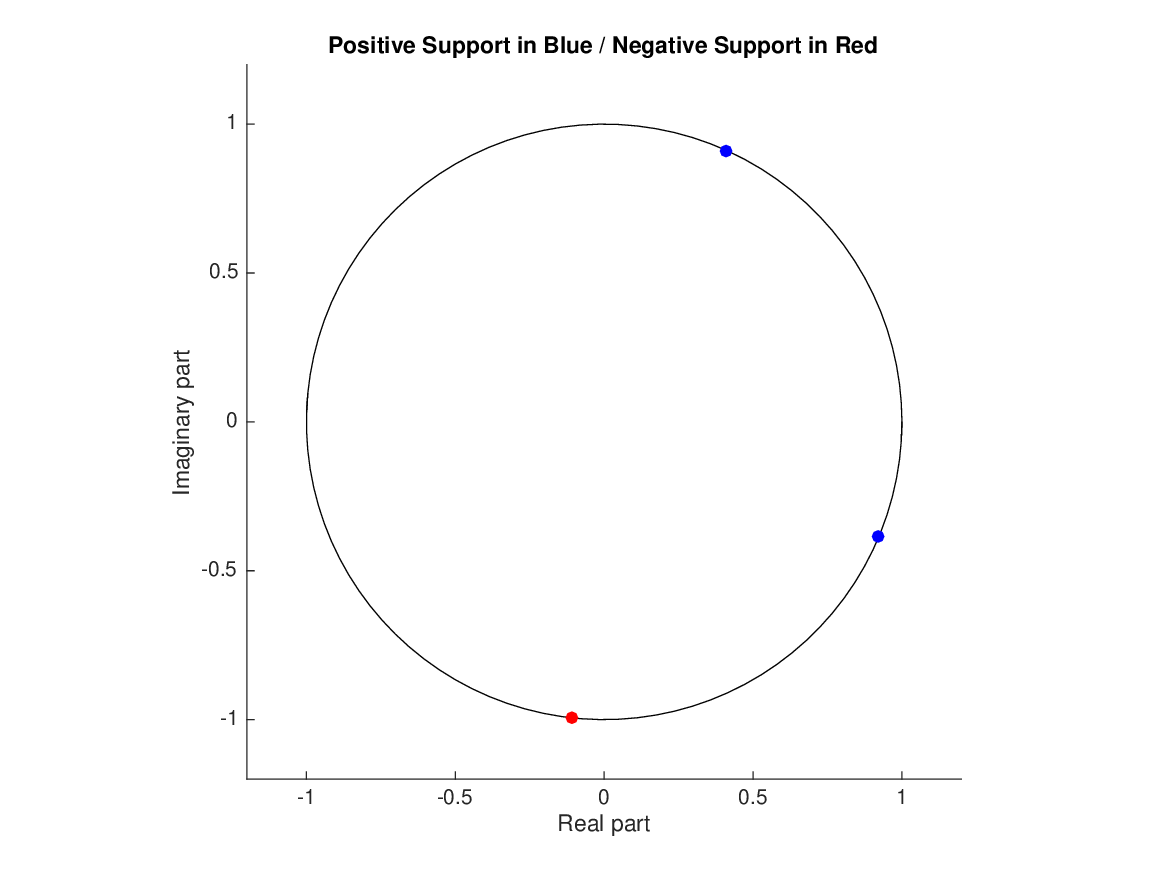}
  \label{fig:test2.1}
\end{minipage}%
\begin{minipage}{.5\textwidth}
  \centering
  \includegraphics[width=1.1\textwidth]{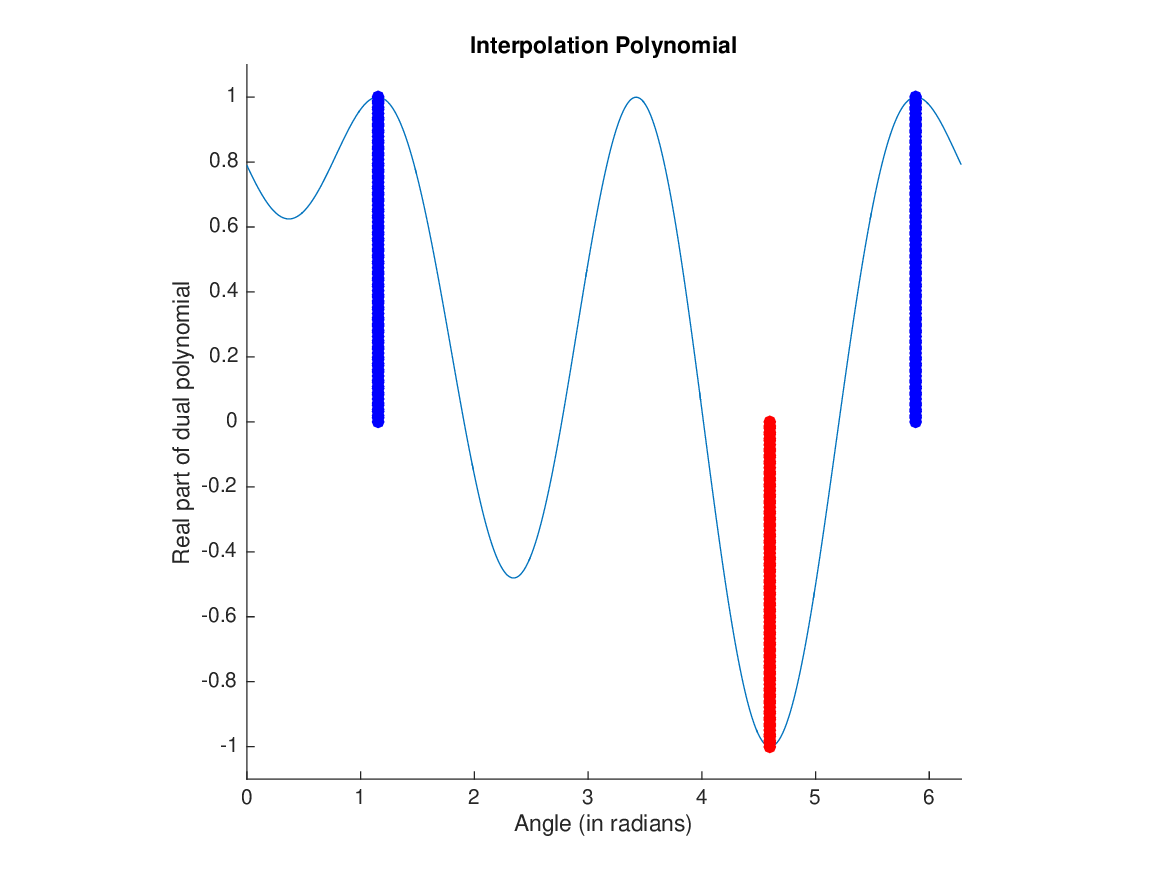}
  \label{fig:test2.2}
\end{minipage}
\caption{Primal-dual solution of super-resolution at order 3 (using multiple loops)}
\label{fig:multiple}
\end{figure}

\subsection{Methodology for comparison}
\label{subsec:Methodology for comparison}
Our methodology for comparing the various approaches is as follows. 
\begin{enumerate}
\item \textbf{Generation of the examples:} We define a random set of ten sparse polynomials with up to ten variables and up to degree ten (first column of Table \ref{tab:music}). We believe that polynomials of higher degree are not realistic and are rarely used in numerical computations. For example, for a polynomial of $n$ variables and $k$ atoms, we generate the exponents $\beta$ of the $k$ monomials $\x^{\beta}$ randomly from $\mathbb{N}^n_d := \{\beta \in \mathbb{N}^n ~|~ \sum_{i=1}^n \beta_i \leqslant d \}$ and the associated non-zero coefficients $g_{\beta}$ are drawns from a uniform distribution in the interval $[-10,10]$.\\
\item \textbf{Results in the noiseless case:} We detect the minimum number of evaluations for each approach to recover the blackbox polynomial in the noiseless case and report the results in Table \ref{tab:music}.
We use evaluations at  the points $e^{\alpha \im}= (e^{\im
  \alpha_{1}},\ldots,e^{\im \alpha_{n}})$ with
$\alpha=(\alpha_{1},\ldots,\alpha_{n}) \in \Z^{n}$ up to a certain
degree $\sum_{k=1}^n |\alpha_k| \leqslant d$. The corresponding number of evaluations and degree $(d)$ are reported in the columns {\em Rigorous LP}, {\em super-resolution}, and {\em Toeplitz Prony} of Table \ref{tab:music}. In {\em Advanced T. Prony}, evaluations are made at different points. Thus, only the first three columns of Table \ref{tab:music} can be compared in presence of noise.\\
\item \textbf{Results in the presence of noise:} For each of the ten polynomials in the list of examples, we determine the maximum degree $d^\text{max}$ for the evaluations $g(e^{i\alpha_1},\hdots,e^{i\alpha_n})$ with $\alpha_1, \hdots , \alpha_n \in \mathbb{Z}$ and $\sum_{i=1}^n |\alpha_k| \leq d^\text{max}$, among {\em Rigorous LP}, {\em super-resolution}, and {\em Toeplitz Prony} in Table \ref{tab:music}. For example, for the first line of Table \ref{tab:music}, that number is $d^\text{max} := 2$ which corresponds to $3$ evaluations in this univariate problem.
As a result, we know that for these evaluations all three approaches return the correct sparse polynomial. We then add uniform noise to those evaluations, i.e.
\begin{equation}
g(e^{\im \alpha_1},\hdots,e^{\im \alpha_n}) + \epsilon_\alpha~, ~~~ \epsilon_\alpha \in \mathbb{C}, ~~~  \Re\epsilon_\alpha,~ \Im \epsilon_\alpha \in [-0.1,+0.1]
\end{equation}
for all $|\alpha_1| + \hdots + |\alpha_n| \leqslant d^\text{max}$ and $\alpha_1, \hdots , \alpha_n \in \mathbb{Z}$.
Next, we run each approach ten times (with new noise every time) and report the average error in Table \ref{tab:noise}. The error is defined as the relative error in percentage of the output polynomial $\hat{g}(\x) = \sum_{\alpha} \hat{g}_\alpha \x^\alpha$ compared with the blackbox polynomial $g(\x) = \sum_{\alpha} g_\alpha \x^\alpha$ using the 2 norm of the coefficients, i.e.
\begin{equation}
100 \times \frac{\sqrt{\sum_{\alpha} (\hat{g}_\alpha - g_\alpha)^2}}{\sqrt{\sum_\alpha g_\alpha^2}}.
\end{equation}
Note that in \textit{Rigorous LP} and \textit{super-resolution} the equalities associated to the evaluations are relaxed to inequalities, a functionality which is not possible in \textit{Toeplitz Prony}. This allows for more robutness. Precisely, in \textit{Rigorous LP}, we replace $\A\x - \b = 0$ by $-0.1 \leqslant \Re (\A\x - \b) \leqslant 0.1$ and $-0.1 \leqslant \Im (\A\x - \b) \leqslant 0.1$, while in \textit{super-resolution} we use a 2-norm ball of radius $0.1 \times \sqrt{2}$ (similar to the technique employed in \cite{candes_towards_2014}).  
\end{enumerate}

\begin{table}[h!]̧
\centering
\begin{tabular}{|c|c|c|c||c|}
\hline
Blackbox & Rigorous & Super & Toeplitz & Advanced\\
Polynomial & LP & Resolution & Prony & T. Prony \\
\hline
$-1.2x^4 + 6.7x^7$ & \hphantom{11}2 (1) & \hphantom{111}3 (2) & \hphantom{111}3 (2) & \hphantom{1}3 \\
\hline
$2.3x^6 + 5.6x^3 -1.5x^2$ & \hphantom{11}4 (3) & \hphantom{111}5 (4) & \hphantom{111}4 (3) & \hphantom{1}4 \\
\hline
$-2.1x^3 + 5.4x^2 -2.0x + 6.2x^5  - 5.2$ & \hphantom{11}5 (4) & \hphantom{111}6 (5) & \hphantom{111}6 (5) & \hphantom{1}6 \\
\hline
$0.8x_1x_2 - x_1x_2^2$ & \hphantom{1}19 (3) & \hphantom{11}31 (4) & \hphantom{11}10 (2) & \hphantom{1}6 \\
\hline
$-5.8x_1^2x_2^2  -8.2x_1^2x_2^3 +5.5x_1^3x_2 + 1.1$ & \hphantom{1}10 (2) & \hphantom{11}19 (3) & \hphantom{11}19 (3) & 13 \\
\hline
$-7.2x_1x_2^2 +1.8x_1^3x_2^2 + 2.6x_1^4x_2^5 + 6.2x_1x_2^5 + 2.5x_1$ & \hphantom{1}10 (2) & \hphantom{11}19 (3) & \hphantom{11}19 (3) & 14 \\
\hline
$-3.5 + 8.1x_1^3x_2x_3$ & \hphantom{11}7 (1) & \hphantom{11}28 (2) & \hphantom{11}28 (2) & \hphantom{1}9 \\
\hline
$-1.2x_1^2x_2^2x_3^3 + 7.3x_1^2x_2 - 2.4x_2$ & \hphantom{1}28 (2) & \hphantom{11}28 (2) & \hphantom{11}28 (2) & 16 \\
\hline
$-6.1x_1^2x_5 +2.5x_2x_4 + 4.8x_3$ & 136 (2) & \hphantom{1}136 (2) & \hphantom{1}136 (2) & 30 \\
\hline
$2.9x_2x_3x_9^4x_{10} - 5.6x_1x_4^2x_7 - 4.1x_3x_5x_6^3x_8$ & N. A. &  1595 (2) & 1595 (2) & 65 \\
\hline
\end{tabular}
\caption{Minimum number of evaluations and degrees without noise (evaluations in the points $(e^{\im\alpha_1}, \hdots , e^{\im \alpha_n})$ for $|\alpha_1| + \hdots + |\alpha_n| \leqslant d$ and $\alpha_1, \hdots , \alpha_n \in \mathbb{Z}$ for the first three columns)}
\label{tab:music}
 \end{table}

 \begin{table}[h!]
 \centering
 \begin{tabular}{|c|c|c|c|}
 \hline
 Blackbox & Rigorous & Super & Toeplitz\\
 Polynomial & LP & Resolution & Prony \\
 \hline
 $-1.2x^4 + 6.7x^7$ & 4.18\% & 1.58\% & 0.61\% \\
 \hline
 $2.3x^6 + 5.6x^3 -1.5x^2$ & 1.94\% & 1.81\% & 0.85\% \\
 \hline
 $-2.1x^3 + 5.4x^2 -2.0x + 6.2x^5  - 5.2$ & 1.47\%  & 1.40\% & 0.69\% \\
 \hline
 $0.8x_1x_2 - x_1x_2^2$ & 3.23\% & 4.84\% & 2.26\% \\
 \hline
 $-5.8x_1^2x_2^2  -8.2x_1^2x_2^3 +5.5x_1^3x_2 + 1.1$ & 1.13\%  & 0.87\% &  1.29\% \\
 \hline
 $-7.2x_1x_2^2 +1.8x_1^3x_2^2 + 2.6x_1^4x_2^5 + 6.2x_1x_2^5 + 2.5x_1$ & 1.23\% &  1.08\% & 6.28\% \\
 \hline
 $-3.5 + 8.1x_1^3x_2x_3$ & 0.79\% & 0.70\% & 0.50\% \\
 \hline
 $-1.2x_1^2x_2^2x_3^3 + 7.3x_1^2x_2 - 2.4x_2$ & 2.19\% & 1.03\% & 1.39\% \\
 \hline
 $-6.1x_1^2x_5 +2.5x_2x_4 + 4.8x_3$ & 0.94\% & 1.15\% & 1.04\% \\
 \hline  
 $2.9x_2x_3x_9^4x_{10} - 5.6x_1x_4^2x_7 - 4.1x_3x_5x_6^3x_8$ & N. A. & 0.47\% & 0.46\% \\
 \hline    
 \end{tabular}
 \caption{Relative error in percentage with uniform noise between $-0.1$ and $0.1$ for the real and imaginary parts on the measurements (evaluations in the points $(e^{i\alpha_1}, \hdots , e^{i \alpha_n})$ for $|\alpha_1| + \hdots + |\alpha_n| \leqslant d$ and $\alpha_1, \hdots , \alpha_n \in \mathbb{Z}$).}
 \label{tab:noise}
 \end{table}

\textbf{The \textit{Advanced T. Prony} column of Table \ref{tab:music}:} the first $r$ exponents $\alpha\in \N^{n}$ are chosen (where $r$ is the number of monomials in the blackbox polynomias) for the monomials indexing the rows and the first $r$ exponents $-\alpha$ with $\alpha\in \N^{n}$ are chosen for indexing the columns of the Toeplitz matrix.  Since $g(e^{-\im\alpha})= \overline{g(e^{\im\alpha})}$, the number of evaluations does not include the conjugate of known values of $g$. The number of monomials $r$ is unknown but one could use {\em Advanced T. Prony} with $r=1,2,\hdots$ successively. We only report the result when setting $r$ to the number of monomials in the blackbox polynomial. Note that in the other approaches in Table \ref{tab:music}, we do \textit{not} assume that the number of monomials is known. Same goes in the presence of noise.

\subsection{Discussion}

Disclaimer: In the sequel, we discuss various advantages and drawbacks of the three methods and of course the resulting conclusions should be interpreted with care as they are biased by the examples that we have considered.

$\bullet$ In the noiseless case, Rigorous LP generally requires the
least number of evaluations compared with super-resolution and
Toeplitz Prony as can be seen in Table \ref{tab:music}. This is a
remarkable situation where the sparse recovery approach (Rigorous LP) is guaranteed to recover the polynomial even if the RIP property does not hold. The classical result on Prony's method is that the number of evaluations to recover the blackbox polynomial is equal to twice the number of monomials in the blackbox polynomial (in the univariate case). Toeplitz Prony goes further: the number of evaluations is equal to the number of monomials plus one (in the univariate case, as explained in Section \ref{subsub:Toeplitz Prony}). For example, the third example in Table \ref{tab:music} requires 6 evaluations and is composed of 5 monomials.

$\bullet$ In terms of certification, in principle, super-resolution has to be applied with enough points ($\geqslant 128$ for $n=1$, $\geqslant 512$ for $n=2$ and more if the separation between the points is small \cite{candes_towards_2014}) to guarantee the existence of a dual certificate polynomial. Moreover, in the multivariate case,  no bound on the order of the SDP relaxation is known to guarantee that the flat extension property is satisfied (rank conditions \eqref{rank+}-\eqref{rank-})\footnote{In a few cases where \eqref{rank+}-\eqref{rank-} are not satisfied, we are still able to a recover polynomial using the algorithm in Section \ref{subsub:Hankel Prony}.}. In contrast, Toeplitz Prony requires evaluations at points $\alpha\in \Z^{n}$ with $|\alpha|\leq r$ where $r$ is at most the number of monomials, in order to recover the decomposition of the sparse polynomial.
  In practice, the experimentations show that a small number of evaluations is sufficient to compute the decomposition in both methods.

$\bullet$ In terms of computational burden, among Rigorous LP, super-resolution, and Toeplitz Prony, the cheapest approach is Toeplitz Prony since it requires only two linear algbebra operations on matrices of size dependent on the number of monomials in the blackbox polynomials. super-resolution entails a heavy computational burden with the semidefinite optimization. Rigorous LP requires the longest setup time because a variable has to be created for each potential monomial in the blackbox polynomial, unlike the two other approaches. In particular, the setup time is too long on a standard laptop for the example with 10 variables (hence N. A. in Table \ref{tab:music} and Table \ref{tab:noise}). Howevever, after the setup step has been performed, computing the LP is fast and reliable.

$\bullet$ Concerning noise, it seems that the three methods perform more and less equally well even with the relatively large noise level that we have selected, namely 0.1 error on the evaluations. This is little bit surprising for the Prony method because it seems to be commonly admitted that Prony is not very robust to noise. This surprising relative robustness may be due to the large threshold $\epsilon = 0.1$ allowed in the rank determination of the SVD decomposition. Indeed, if we select a smaller threshold, we observe degradation of the results for Prony (and super resolution which relies on Prony for the extraction step after the optimization step). See table \ref{tab:thres} below.

\begin{table}[h!]
\centering
\begin{tabular}{|c|c|c|c|c|c|c|c|c|}
\hline
Blackbox & \multicolumn{4}{|c|}{Super Resolution} & \multicolumn{4}{c|}{Toeplitz Prony}\\
\cline{2-9}
Polynomial & $10^{-1}$ & $10^{-2}$ & $10^{-3}$ & $10^{-4}$ & $10^{-1}$ & $10^{-2}$ & $10^{-3}$ & $10^{-4}$ \\
\hline
$-1.2x^4 + 6.7x^7$ &   1.58\% &    \hphantom{1}1.58\% &    \hphantom{11}1.58\% &    \hphantom{11}1.58\% &   0.79\% &    \hphantom{11}0.79\% &   \hphantom{11}0.79\% &   \hphantom{11}0.79\% \\
\hline
 $2.3x^6 + 5.6x^3 -1.5x^2$ &   1.80\% &    \hphantom{1}1.80\% &    \hphantom{11}1.80\% &    \hphantom{11}1.80\% &    0.91\% &    \hphantom{11}1.15\% &    \hphantom{11}1.15\% &    \hphantom{11}1.15\% \\
\hline
$-2.1x^3 + 5.4x^2 -2.0x + 6.2x^5  - 5.2$ &   1.42\% &    \hphantom{1}1.44\% &   \hphantom{11}1.44\% &    \hphantom{11}1.44\% &    0.68\% &    \hphantom{11}0.68\% &    \hphantom{11}0.68\% &  \hphantom{11}0.68\% \\
\hline
 $0.8x_1x_2 - x_1x_2^2$ &    4.69\% &   17.25\% &   \hphantom{1}12.49\% &   \hphantom{1}25.50\% &    5.35\% &  137.90\% &  165.79\% &  226.29\% \\
\hline 
 $-5.8x_1^2x_2^2  -8.2x_1^2x_2^3 +5.5x_1^3x_2 + 1.1$ &    0.93\% &    \hphantom{1}2.10\% &   \hphantom{1}42.69\% &   \hphantom{1}47.50\% &    1.22\% &   \hphantom{1}60.36\% &   \hphantom{1}53.86\% &   \hphantom{1}42.47\% \\
\hline
$-7.2x_1x_2^2 +1.8x_1^3x_2^2 + 2.6x_1^4x_2^5 + \hdots$ &    1.00\% &   17.56\% &   \hphantom{1}43.39\% &   \hphantom{1}67.55\% &    8.45\% &   \hphantom{1}89.45\% &   \hphantom{1}45.82\% &   \hphantom{1}33.87\% \\
\hline
 $-3.5 + 8.1x_1^3x_2x_3$ &    0.78\% &    \hphantom{1}0.78\% &   \hphantom{1}72.17\% &   \hphantom{1}62.02\% &    0.41\% &   \hphantom{1}96.50\% &   \hphantom{1}77.42\% &   \hphantom{1}83.27\% \\
\hline
$-1.2x_1^2x_2^2x_3^3 + 7.3x_1^2x_2 - 2.4x_2$ &  1.09\% &   18.55\% &  \hphantom{1}86.55\% &   \hphantom{1}83.45\% & 3.59\% &   \hphantom{1}66.02\% &   \hphantom{1}39.77\% &   \hphantom{1}55.15\% \\
\hline
 $-6.1x_1^2x_5 +2.5x_2x_4 + 4.8x_3$ &    0.85\% &   \hphantom{1}24.60\% &  105.84\% &  140.05\% &    1.68\% &  130.39\% &   \hphantom{1}86.77\% &   \hphantom{1}79.96\% \\
\hline
 $2.9x_2x_3x_9^4x_{10} - 5.6x_1x_4^2x_7 + \hdots $ &   0.54\%  &   \hphantom{1}0.54\% &  136.20\%  & 146.93\%  &   6.45\% &  251.17\% &  119.65\% &  257.64\% \\
\hline
\end{tabular}
\caption{Same experiments as in Table \ref{tab:noise} but with four different values of the rank threshold $\epsilon = 10^{-1},10^{-2},10^{-3},10^{-4}$.}
\label{tab:thres}
\end{table}

\section{Another efficient (\textit{a priori} heuristic) approach}

In sequel we propose still use the super-resolution hierarchy \eqref{sdp-Toeplitz} but now by restricting the evaluations at points $\varphi^{\alpha}$ with

\begin{equation}
\alpha \in\,\mathscr{A}^2_{\fc}\,:=\,\{\alpha \in \mathbb{N}^n ~|~
\|\alpha\|_1 \leqslant \fc \} \subset \mathscr{A}^1_{\fc} \subset
\mathscr{A}_{\fc}.
\end{equation}

This is illustrated in Figure \ref{fig:eval2}. This restriction is first inspired by the fact that Hankel Prony descibed in Section \ref{subsub:Hankel Prony} is guaranteed to work using only those evaluations. There is another more general inspiration coming from two mathematical results.

\begin{figure}[!h]
	\centering
	\includegraphics[width=.6\textwidth]{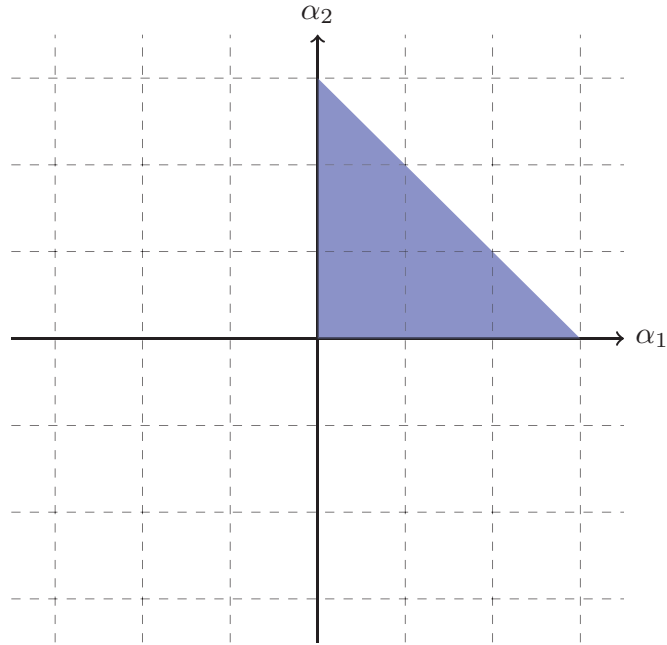}
	\caption{Evaluations at $\alpha$ with $\alpha_1+\alpha_2 \leqslant 3$ and $\alpha_1,\alpha_2 \in \mathbb{N}$.}
\label{fig:eval2}
\end{figure}

We provide a result valid in full generality for atomic measures (with finitely many atoms) which indeed suggests
that in practice it may suffice to make evaluations at $\alpha \in \mathbb{N}$ (instead of $\alpha \in \mathbb{Z}$). The resulting semidefinite programs have Toeplitz matrices of the same dimension but include much less linear moment constraints. With 10 variables, the first order semidefinite program of the hierarchy entails matrix variables of size $11 \times 11$ and \textit{only} $11$ linear equalities (instead of 56)! The second order relaxation entails matrix variables of size $66 \times 66$ and $66$ linear equalities (instead of 1,596), and the third relaxation entails matrix variables of size $286 \times 286$ and $286$ linear equalities (instead of 21,691), and so on.
We first remind the reader of a well-known result.

\begin{prop}[Consequence of Stone-Weiestrass]
\label{prop-Stones-W}
Let $(y_{\alpha,\beta})_{\alpha,\beta \in \mathbb{N}^n}$ denote a multi-indexed sequence of complex numbers and let $K \subset \mathbb{C}^n$ denote a compact set. If there exists a complex-valued finite Borel measure $\mu$ supported on $K$ such that 
\begin{equation}
\label{eq:mm}
y_{\alpha,\beta} = \int_{K} z^\alpha \bar{z}^\beta d\mu ~ , ~~~ \forall \alpha , \beta \in \mathbb{N}^n,
\end{equation}
then $\mu$ is the unique complex-valued finite Borel measure to satisfy \eqref{eq:mm}.
\end{prop}
\begin{proof}
Consider another such measure $\hat{\mu}$. Then
\begin{equation}
\int_{K} z^\alpha \bar{z}^\beta d(\mu-\hat{\mu}) = 0~ , ~~~ \forall \alpha , \beta \in \mathbb{N}^n.
\end{equation}
Thanks to the complex Stone-Weiestrass Theorem,
\begin{equation}
\int_{\mathbb{C}^n} \varphi d(\mu-\hat{\mu}) = 0~ , ~~~ \forall \alpha , \beta \in \mathbb{N}^n.
\end{equation}
for all function $\varphi : \mathbb{C}^n \longrightarrow \mathbb{C}$ continuous with respect to the sup-norm $\|\varphi\|_{\infty} := \sup_{z \in K} |\varphi(z)|$. Therefore $\mu=\hat{\mu}$.
\end{proof}
In practice, whether it be interpolation or optimization, we are generally interested in {\it atomic} measures with finitely many atoms (in short, atomic measures in the sequel). 
The next result establishes that for such atomic measures we do not have to care about conjugates, which in view of
Proposition \ref{prop-Stones-W}, we find somewhat counter-intuitive.
\begin{lemma}
\label{lemma:atom}
Let $(y_{\alpha})_{\alpha\in \mathbb{N}^n}$ denote a multi-indexed sequence of complex numbers. If there exists an atomic complex-valued measure $\mu$ such that 
\begin{equation}
\label{eq:mmpb}
y_{\alpha} = \int_{\mathbb{C}^n} z^\alpha d\mu ~ , ~~~ \forall \alpha \in \mathbb{N}^n,
\end{equation}
then $\mu$ is the unique atomic measure to satisfy \eqref{eq:mmpb}.
\end{lemma}
\begin{proof}
Let us write the measure $\mu$ as
\begin{equation} \mu = \sum_{k=1}^r w_k \delta_{\xi_{k}} \end{equation}
where $d\in \mathbb{N}$, and $w_1,\hdots,w_r \in \mathbb{C}\setminus \{0\}$, and $\xi_{k},\hdots , \xi_{r} \in \mathbb{C}^n$.\\\\
Consider another atomic measure $\hat{\mu}$ that satisfies \eqref{eq:mmpb}, of the form
\begin{equation}
  \hat{\mu} = \sum_{k=1}^{\hat{r}} \hat{w}_k \delta_{\hat{\xi}_{k}}
\end{equation}
where $\hat{r}\in \mathbb{N}$, and $\hat{w}_1,\hdots,\hat{w}_{\hat{r}} \in \mathbb{C}\setminus \{0\}$, and $\hat{\xi}_{1},\hdots , \hat{\xi}_{\hat{r}} \in \mathbb{C}^n$.\\\\
Consider the following truncated Hankel matrix 
\begin{equation}
H_k(y) = ( y_{\alpha+\beta} )_{|\alpha|,|\beta|\leqslant k}
\end{equation}
where $|\alpha| := \alpha_1 + \hdots + \alpha_n$. Thanks to Lemma \ref{lemma:ind}, its rank is equal to $r$ when $k \geqslant r-1$ and it is equal to $\hat{r}$ when $k \geqslant \hat{r}-1$. Thus $r = \hat{r}$. Moreover, when $k \geqslant r = \hat{r}$, Lemma \ref{lemma:range} implies that 
\begin{equation}
\label{eq:span}
\text{span} \{ ~ v_k ( \xi_{1} ) ~ , ~ \hdots ~ , ~ v_k( \xi_{r} ) ~ \} 
~~=~~
\text{span} \{ ~ v_k ( \hat{\xi}_{1} ) ~ , ~ \hdots ~ , ~ v_k( \hat{\xi}_{r} ) ~ \} 
\end{equation}
where $v_k(z) = (z^\alpha)_{|\alpha|\leqslant k}$. \\\\
We now reason by contradiction. Assume that one of the atoms of $\hat{\mu}$, say $\hat{\xi}_{1}$, is distinct from the atoms of $\mu$. Hence $\hat{\xi}_{1},\xi_{1},\hdots , \xi_{r}$ are $r+1$ distinct points of $\mathbb{C}^n$. Lemma \ref{lemma:ind} implies that 
\begin{equation}
v_k ( \hat{\xi}_{1} ) ~,~ v_k ( \xi_{1} ) ~ , ~ \hdots ~ , ~ v_k( \xi_{r} )
\end{equation}
are linearly independent vectors if $k \geqslant r$. This contradicts equation \eqref{eq:span}. The atoms of $\mu$ and $\hat{\mu}$ thus coincide. Their weights satisfy
\begin{equation}
(w_1-\hat{w}_1) v_k ( \xi_{1} ) ~ + ~ \hdots ~ + ~ (w_d-\hat{w}_r) v_k( \xi_{r} ) = 0.
\end{equation}
Again, thanks to Lemma \ref{lemma:ind}, the vectors are linearly independent if $k \geqslant r-1$, thus $w_1-\hat{w}_1 = \hdots = w_d-\hat{w}_r = 0$. This terminates the proof.
\end{proof}

\textbf{Numerical experiments:}

Below, we replicate the experiments of Section \ref{subsec:Methodology for comparison} (with and without noise) but now we make evaluations in $\alpha_1,\hdots,\alpha_n \in \mathbb{N}$ instead of $\alpha_1,\hdots,\alpha_n \in \mathbb{Z}$. 

\begin{table}[h!]̧
\centering
\begin{tabular}{|c|c|c|c||c|}
\hline
Blackbox & Rigorous & Super & Hankel & Advanced\\
Polynomial & LP & Resolution & Prony & H. Prony \\
\hline
$-1.2x^4 + 6.7x^7$ & \hphantom{1}2 (1) & \hphantom{1}3 (2) & \hphantom{11}4 (3) & \hphantom{1}4 \\
\hline
$2.3x^6 + 5.6x^3 -1.5x^2$ & \hphantom{1}4 (3) & \hphantom{1}5 (4) & \hphantom{11}6 (5) & \hphantom{1}6 \\
\hline
$-2.1x^3 + 5.4x^2 -2.0x + 6.2x^5  - 5.2$ & \hphantom{1}5 (4) & \hphantom{1}6 (5) & \hphantom{1}10 (9) & 10 \\
\hline
$0.8x_1x_2 - x_1x_2^2$ & 10 (3) & 15 (4) & \hphantom{1}10 (3) & \hphantom{1}7 \\
\hline
$-5.8x_1^2x_2^2  -8.2x_1^2x_2^3 +5.5x_1^3x_2 + 1.1$ & 10 (3) & 15 (4) & \hphantom{1}21 (5) & 15 \\
\hline
$-7.2x_1x_2^2 +1.8x_1^3x_2^2 + 2.6x_1^4x_2^5 + 6.2x_1x_2^5 + 2.5x_1$ & 10 (3) & 15 (4) & \hphantom{1}21 (5) & 18 \\
\hline
$-3.5 + 8.1x_1^3x_2x_3$ & 10 (2) & 10 (2) & \hphantom{1}20 (3) & 10 \\
\hline
$-1.2x_1^2x_2^2x_3^3 + 7.3x_1^2x_2 - 2.4x_2$ & 20 (3) & 20 (3) & \hphantom{1}20 (3) & 16 \\
\hline
$-6.1x_1^2x_5 +2.5x_2x_4 + 4.8x_3$ & 21 (2) & 21 (2) & \hphantom{1}56 (3) & 28 \\
\hline
$2.9x_2x_3x_9^4x_{10} - 5.6x_1x_4^2x_7 - 4.1x_3x_5x_6^3x_8$ & 66 (2) & 66 (2) & 286 (3) & 58 \\
\hline
\end{tabular}
\caption{Minimum number of evaluations and degrees without noise (evaluations in the points $(e^{\im\alpha_1}, \hdots , e^{\im \alpha_n})$ for $\alpha_1 + \hdots + \alpha_n \leqslant d$ and $\alpha_1, \hdots , \alpha_n \in \mathbb{N}$ for the first three columns)}
\label{tab:musicbis}
 \end{table}

 \begin{table}[h!]
 \centering
 \begin{tabular}{|c|c|c|c|}
 \hline
 Blackbox & Rigorous & Super & Hankel\\
 Polynomial & LP & Resolution & Prony \\
 \hline
 $-1.2x^4 + 6.7x^5$ & \hphantom{1}2.32\% & \hphantom{1}1.66\% & \hphantom{1}0.97\% \\
 \hline
 $2.3x^6 + 5.6x^3 -1.5x^2$ & \hphantom{1}1.71\% & \hphantom{1}2.31\% & \hphantom{1}3.33\% \\
 \hline
 $-2.1x^3 + 5.4x^2 -2.0x + 6.2x^5  - 5.2$ & \hphantom{1}0.80\%  & \hphantom{1}1.64\% & \hphantom{1}2.89\% \\
 \hline
 $0.8x_1x_2 - x_1x_2^2$ & 14.91\% & 11.03\% & 52.14\% \\
 \hline
 $-5.8x_1^2x_2^2  -8.2x_1^2x_2^3 +5.5x_1^3x_2 + 1.1$ & \hphantom{1}0.73\%  & \hphantom{1}1.01\% &  \hphantom{1}2.13\% \\
 \hline
 $-7.2x_1x_2^2 +1.8x_1^3x_2^2 + 2.6x_1^4x_2^5 + 6.2x_1x_2^5 + 2.5x_1$ & \hphantom{1}1.19\% &  12.30\% & \hphantom{1}2.67\% \\
 \hline
 $-3.5 + 8.1x_1^3x_2x_3$ & \hphantom{1}0.82\% & \hphantom{1}1.32\% & \hphantom{1}0.93\% \\
 \hline
 $-1.2x_1^2x_2^2x_3^3 + 7.3x_1^2x_2 - 2.4x_2$ &  \hphantom{1}3.29\% & \hphantom{1}2.13\% & 16.99\% \\
 \hline
 $-6.1x_1^2x_5 +2.5x_2x_4 + 4.8x_3$ & \hphantom{1}2.90\% & \hphantom{1}1.64\% & \hphantom{1}6.74\% \\
 \hline  
 & 107.87\% (1) & 161.36\% (1) & 134.87\% (1) \\
 $2.9x_2x_3x_9^4x_{10} - 5.6x_1x_4^2x_7 - 4.1x_3x_5x_6^3x_8$ & \hphantom{111}N.A. (2) & \hphantom{11}2.12\% (2) & 134.69\% (2) \\
 & \hphantom{111}N.A. (3) & \hphantom{111}N.A. (3) & \hphantom{11}0.57\% (3) \\
 \hline    
 \end{tabular}
 \caption{Relative error in percentage with uniform noise between $-0.1$ and $0.1$ for the real and imaginary parts on the measurements (evaluations in the points $(e^{i\alpha_1}, \hdots , e^{i \alpha_n})$ for $\alpha_1 + \hdots + \alpha_n \leqslant d$ and $\alpha_1, \hdots , \alpha_n \in \mathbb{N}$).}
\label{tab:noisebis}
 \end{table}

In Table \ref{tab:musicbis}, in the column {\em Advanced H. Prony}, the first $r$ exponents $\alpha\in \N^{n}$ are chosen for the monomials indexing the rows and columns of the Hankel matrix, where $r$ is the number of terms in the blackbox polynomial $g$. As in Section \ref{subsec:Methodology for comparison}, for the first three columns of Table \ref{tab:musicbis} and Table \ref{tab:noisebis}, we do not assume anything to be known about the blackbox polynomial expect for the number of variables and an upper bound on the degree (i.e. 10).

In the presence of noise, the optimization step of super-resolution (before the second step of extraction) seems to behave as an efficient filter as it indeed reduces the error compared with Hankel Prony in 7 out of the 9 comparable instances of Table \ref{tab:noisebis}.
However, sometimes, the semidefinite program does not provide a good output. Indeed, in the sixth example, among the ten trials
 there are two trials where the solver runs into numerical issues, which explains the large error of $12.30\%$.

%

\section{Conclusion}

We have addressed the  sparse polynomial interpolation problem
with three different approaches: sparse recovery, super resolution, and Prony's method. 
The common denominator of the three approaches is our view of a polynomial as a signed atomic measure
where the atoms correspond to monomials and the weights to coefficients. Then, on the one hand we can invoke directly results from (discrete) super-resolution
theory \`a la Cand\`es \& Fernandez-Granda \cite{candes_towards_2014} to show that the unknown 
black box polynomial is the unique solution of a certain LP on a measure space and also
the unique solution of finite-dimensional linear program. On the other hand, invoking Kunis et al. \cite{kunis_multivariate_2016}
Prony's method can also be applied. To the best of our knowledge this unifying view of sparse interpolation is new and makes 
the numerical comparison of the three methods very natural. In our preliminary numerical experiments :
\begin{itemize}
\item Prony's method works well and better than expected in the presence of noise. 
\item Super-resolution acts in two steps: a first optimization step and then an extraction procedure applied
to the optimal solution. The latter step is nothing less than Prony's method. We find that this optimization step
sometimes helps significantly in the presence of noise.
\item LP-sparse recovery also works well but its set-up time is quite limiting.
\end{itemize}

\section*{Acknowledgement}
The work of the first two authors was funded by the European Research Council (ERC) under the European Union's
Horizon 2020 research and innovation program (grant agreement 666981 TAMING).

\section*{Appendix}

\begin{lemma} 
\label{lemma:ind}
If $z^{(1)}, \hdots , z^{(d)}$ are distinct points of $\mathbb{C}^n$, then $v_{d-1}(z^{(1)}), \hdots , v_{d-1}(z^{(d)})$ are linearly independent vectors, where $v_d(z) := (z^\alpha)_{|\alpha\leqslant d}$.
\end{lemma}
\begin{proof}
Consider some complex numbers $c_1, \hdots, c_d$ such that 
\begin{equation}
\label{eq:dependency}
\sum_{k=1}^d c_k (z^{(k)})^\alpha = 0 ~,~~~ \forall |\alpha|\leqslant d-1.
\end{equation}
Given $1\leqslant l \leqslant d$, define the Lagrange interpolation polynomial 
\begin{equation}
L^{(l)}(z) := \prod_{\scriptsize \begin{array}{c} 1 \leqslant k \leqslant d \\ k \neq l \end{array}}
\frac{ z_{i(k)} - z^{(k)}_{i(k)} }{ z^{(l)}_{i(k)} - z^{(k)}_{i(k)} } 
\end{equation}
 where $i(k) \in \{1,\hdots, n \}$ is an index such that $z^{(k)}_{i(k)} \neq z^{(l)}_{i(k)}$. It satisfies $L^{(l)}(z^{(k)}) = 1$ if $k = l$ and $L^{(l)}(z^{(k)}) = 0$ if $k \neq l$. The degree of $L^{(l)}(z) =: \sum_{\alpha} L^{(l)}_\alpha z^\alpha$ is equal to $d - 1$. Thus we may multiply the equation in \eqref{eq:dependency} by $L^{(l)}_\alpha$ to obtain
\begin{equation}
\sum_{k=1}^d c_k ~ L^{(l)}_\alpha (z^{(k)})^\alpha = 0 ~,~~~ \forall |\alpha|\leqslant d-1.
\end{equation}
Summing over all $|\alpha|\leqslant d-1$ yields $\sum_{k=1}^d c_k ~ L^{(l)}(z^{(k)}) = c_l = 0$. 
\end{proof}

\begin{lemma}
\label{lemma:range}
\normalfont
\textit{If $u_1,\hdots,u_d \in \mathbb{C}^n$ are linearly independent, and $c_1,\hdots,c_d \in \mathbb{C}\setminus \{0\}$, then} $\mathcal{R}(\sum_{i=1}^d c_i u_i u_i^T) = \mathcal{R}(\sum_{i=1}^d c_i u_i u_i^*) = \text{span} \{u_1,\hdots,u_d\}$ where $\mathcal{R}$ denotes the range.
\end{lemma}
\begin{proof}
If $z \in \mathbb{C}^n$, then $(\sum_{i=1}^d c_i u_i u_i^T)z = \sum_{i=1}^d (c_i  u_i^T z) u_i \in \text{span} \{u_1,\hdots,u_d\}$ and $(\sum_{i=1}^d c_i u_i u_i^*)z = \sum_{i=1}^d (c_i  u_i^* z) u_i \in \text{span} \{u_1,\hdots,u_d\}$. Conversly, an element of the span $\sum_{i=1}^d \lambda_i u_i$ with $\lambda_1,\hdots,\lambda_n \in \mathbb{C}$ belongs to the the range of $\sum_{i=1}^d c_i u_i u_i^T$ if there exists $z\in\mathbb{C}^n$ such that 
$$ \sum_{i=1}^d \lambda_i u_i = \left( \sum_{i=1}^d c_i  u_i u_i^T \right) z  $$ 
which is equivalent to each of the next three lines:
\begin{equation} \sum_{i=1}^d [ \lambda_i - (c_i  u_i^T z) ] u_i = 0, \end{equation}
\begin{equation}  \lambda_i = (c_i  u_i)^T z ~,~ i=1,\hdots,d,\end{equation}
\begin{equation} \lambda =  ( c_1 u_1 \hdots c_d u_d )^T z. \end{equation}
Since $( c_1 u_1 \hdots c_d u_d ) \in \mathbb{C}^{n \times d}$ has rank $d$, its transpose has rank $d$. Thus there exists a desired $z \in \mathbb{C}^n$. Likewise, $\sum_{i=1}^d \lambda_i u_i$ belongs to the the range of $\sum_{i=1}^d c_i u_i u_i^*$ if there exists $z\in\mathbb{C}^n$ such that 
$$  \lambda_i = (c_i  u_i)^* z ~,~ i=1,\hdots,d. $$
Since $( c_1 u_1 \hdots c_d u_d ) \in \mathbb{C}^{n \times p}$ has rank $d$, its conjugate transpose has rank $d$. Thus there exists a desired $z \in \mathbb{C}^n$.
\end{proof}




\end{document}